\documentclass[11pt,reqno]{amsart}

\renewcommand{\Im}{\operatorname{Im}}

\usepackage{amssymb}
\usepackage{amsthm}
\usepackage{amsxtra}

\usepackage{tikz}
\usepackage{fancyhdr}
\usepackage{caption}

\usepackage{graphicx}
\usepackage{color}
\usepackage{amsmath}
\usepackage{listings} 
\usepackage{float}
\usepackage{cases} 
\usepackage{threeparttable}
\usepackage{dcolumn}
\usepackage{multirow}
\usepackage{booktabs}
\usepackage{comment}
\usepackage{varwidth}
\usepackage{mathtools}
\usepackage{subcaption}

\usepackage{cancel}

\newcommand{\R}{\mathbb R}

\newcommand\norm[1]{\left\lVert#1\right\rVert}
\newcommand{\vertiii}[1]{{\left\vert\kern-0.25ex\left\vert\kern-0.25ex\left\vert #1 
    \right\vert\kern-0.25ex\right\vert\kern-0.25ex\right\vert}}

\newtheorem{theorem}{Theorem}[section]

\newtheorem{proposition}[theorem]{Proposition}
\newtheorem{lemma}[theorem]{Lemma}
\newtheorem{corollary}[theorem]{Corollary}

\theoremstyle{remark}

\newtheorem{remark}[theorem]{Remark}

\numberwithin{equation}{section}

\title[Blow-up in 1D combined NLS]{Blow-up criteria for the 1D NLS with combined nonlinearities}

\author[Alex D. Rodriguez]{Alex D. Rodriguez}
\address{Department of Mathematics  \& Statistics\\Florida International University,  Miami, FL, USA}
\curraddr{}
\email{arodr1128@fiu.edu}



\begin{document}

\begin{abstract}
    In this paper we develop two different types of criteria for the finite time blow-up solutions to the combined nonlinear Schr\"odinger equation in 1D. The first one is a negative energy criterion developed for triple combined nonlinearity and then for a (convergent) infinite sum of combined nonlinearities, including an exponential nonlinearity. The second one is a positive energy criterion for the double nonlinearity with the defocusing-focusing coefficients. We also provide examples of initial data that satisfy either of the criteria.
\end{abstract}

\maketitle

\section{Introduction}\label{S:Intro}

We study blow-up criteria for the initial value problem for the 1D nonlinear Schr\"odinger equation
\begin{equation}{\label{ClassicNLS}}
    \begin{cases}
        &iu_{t} + u_{xx} + \mathcal{N}(u)u = 0, \quad (x,t) \in \R\times \R,\\
        &u(x,0) = u_{0}, \\
    \end{cases}
\end{equation}
where $\mathcal{N}(u)$ is a nonlinear term, typically presented as $\mathcal{N}(u) = \epsilon|u|^\alpha$ with $\epsilon \in
\mathbb{C}$ and $\alpha > 0$. In this paper, we consider a linear combination of several such nonlinear terms, including an {\it infinite} sum. With the single nonlinearity, the equation \eqref{ClassicNLS} formally conserves mass and energy 
\begin{equation*}\label{E:Mass_single}
	M[u](t) = \int_{-\infty}^\infty|u(x,t)|^2 \, dx  = M[u_0],
\end{equation*}

\begin{equation*}\label{E:Energy_single}
	E[u](t) = \frac{1}{2}\int_{-\infty}^\infty|u_x(x,t)|^2dx - \frac{\epsilon}{\alpha + 2}\int_{-\infty}^\infty|u(x,t)|^{\alpha + 2} \, dx = E[u_0].
\end{equation*}

We say a solution $u(x,t)$ \textit{blows-up} in finite time, if for a time $0<T^*<\infty$, $\lim_{t\to T^*} \| u_x(t) \|_{L^2} = \infty$. To study this behavior, we use the so-called \textit{virial identity}, for which we first define a variance 
\begin{equation}\label{E:Variance_single}
    V[u](t) =  \int_{-\infty}^\infty|xu(x,t)|^2 \, dx =\|xu\|_{L^2}^2.
\end{equation}
For a finite variance $V[u](t) < \infty$, the solution $u(x,t)$ satisfies the virial identity 
\begin{align}\label{E:singleVar}
    V_{tt}[u](t) = 4\alpha E[u] + 2(4-\alpha) \| u_x(t) \|_{L^2}^2.
\end{align}
(For results on infinite variance, see Ogawa-Tsutsumi \cite{OT1991} for a radial case and negative energy, Holmer-Platte-Roudenko \cite{HRP2010} for a radial case and positive energy.)

From \eqref{E:singleVar}, we obtain the following cases: 
\begin{enumerate}
    \item[(1)] if $\alpha \geq 4$ and $E[u] \leq 0$, then $V_{tt} \leq 0$ and finite time blow up can be deduced, see \cite{VPT1971,Glassey1977};
    \item[(2)] if $\alpha > 4$ and $E[u] > 0$, then the sign of $V_{tt}$ is unclear and further study is required, see \cite{HR2007, Lushnikov1995, HRP2010, DR2015};
    \item[(3)] if $\alpha < 4$, then \eqref{ClassicNLS} is mass-subcritical and is globally well-posed, e.g., in $H^1(\R)$, \cite{GV1979, Caz-book}.
\end{enumerate}

For case (3), there are a variety of well-posedness results in the literature with many local and global results available in $H^s$ for $s\leq 1$ (for example, see \cite{GV1979}, \cite{GV1985}, \cite{K1987}, \cite{T1987}, \cite{B1993}, \cite{Iteam2002}). Global behavior for the NLS of this type is well understood for different powers of a single nonlinearity (in higher dimensions as well), see, for instance, books \cite{LPIntroToNDE2014}, \cite{Caz-book}, \cite{TaoBook}, \cite{SS1999}, \cite{F2015} for various local and global results, including scattering, blow-up and long-term dynamics for the focusing NLS.

What we have yet to introduce is the case of combined nonlinear terms,  $\mathcal{N}(u)=\sum_{j}\lambda_j|u|^{\alpha_j}$, as results for many nonlinear terms are relatively sparse in the literature. Well-posedness, scattering and negative energy blow-up for {\it two} nonlinear terms have been studied in depth in \cite{TVZ2005}. Some of the physically relevant combined nonlinearities such as cubic-quintic have been investigated more thoroughly. In that case (of the cubic-quintic nonlinearity), unstable solutions described as \textit{bubbles} in Bose condensates were studied in \cite{BM1988}; scattering was studied in 3D in \cite{KMV2021} and in the 2D radial case in \cite{M2021}; the orbital stability and scattering in two- and three-dimensional settings in \cite{CS2021} and the instability of solitary waves in 2D and 3D in \cite{CKS2023}. 
For more than two nonlinear terms, well-posedness, solitary waves and other global results are investigated very little, see recent works on {\it triple} nonlinearities \cite{Liu2021}, \cite{Phan2024}. For a very recent result, on the well-posedness and scattering in weighted Sobolev spaces for an {\it infinite} sum of nonlinear terms for the NLS \eqref{ClassicNLS} equation, see \cite{RRR2024}. This paper seeks to be a natural extension to \cite{RRR2024}, investigating blow-up, since such solutions were observed numerically in \cite[Section 6]{RRR2024}. 

The main goal of this paper is to obtain blow-up criteria for solutions to the NLS equation \eqref{ClassicNLS} with {\it finitely} or {\it infinitely} many combined nonlinear terms, including some known analytical nonlinearities, such as exponential. We obtain criteria for negative energy and finite variance in that case. We also explore blow-up criteria for positive energy solutions and obtain a blow-up criterion for the NLS with a double combined potential, i.e., two nonlinear terms. 

We utilize two different methods for developing criteria for finite time blow-up solutions. The {\it first} one is a classical method developed by Vlasov-Petrishchev-Talanov \cite{VPT1971} and by Glassey \cite{Glassey1977}, where the convexity argument relies on the energy of the NLS solution being negative. This method (with various interpolations) was also used by Tao-Visan-Zhang in \cite{TVZ2005} for the combined (two terms) NLS equation, we adapt it here first for the three nonlinear terms and then for an exponential nonlinearity, written as an infinite sum of combined nonlinearities. Motivated by the exponential nonlinearity, we then generalize that case to a finite or infinite sum of \textit{focusing} nonlinearities. The {\it second} method was proposed by Lushnikov \cite{Lushnikov1995}, where he observed that writing the solution in polar representation allows one to include the variance in the second derivative of the virial identity, and using some physical mechanics interpretation, deduce a blow-up criteria for positive energy. This method was further developed by Holmer-Platte-Roudenko in \cite{HRP2010} into an actual criterion and modifying it via some new interpolation inequality (with finding a sharp constant), obtaining another criterion for finite time blow-up in the cubic NLS in 3D, and later in a more general context by Duyckaerts-Roudenko in \cite{DR2015}.
\smallskip

We begin by presenting the first two theorems, which rely on the assumption that the energy is negative (for a precise definition of the energy in this case, see Section \ref{S:MEV}, also for conciseness, we abbreviate the mass $M[u_0]$ and the energy $E[u_0]$ as $M$ and $E$, correspondingly).

\begin{theorem}[Negative energy blow-up for three nonlinear terms]\label{NegE_Bup_3term}
Let $u_0 \in H^1(\mathbb{R})$ such that $V[u_0] < \infty$. Let $u(x,t)$ be a solution to \eqref{ClassicNLS} with the given initial condition $u_0$ and the nonlinearity
$$
\mathcal{N}(u) = -\lambda_1 |u|^{\alpha_1}-\lambda_2 |u|^{\alpha_2} -\lambda_3 |u|^{\alpha_3},
$$
such that $\lambda_3 < 0$, $\alpha_3>4$ (i.e., the largest nonlinearity is always focusing and greater than 4), $\lambda_1, \lambda_2 \in \R$ and $\alpha_1,\alpha_2 \in \R^+$. 
For $\theta_1 \in (\frac4{\alpha_3},1)$ and $\theta_2 \in (0,1)$, let $\delta>0$ satisfy conditions \eqref{E:delta-2B}, \eqref{E:delta-2C}, \eqref{E:delta-3C}, \eqref{E:delta-4A}, \eqref{E:delta-4B}, \eqref{E:delta-4C}, correspondingly, for a given range of $\lambda$'s and $\alpha$'s as indicated in the table below with the respective condition (in the third column) satisfied, where $C(\delta)$ is given by \eqref{E:C_delta}:
\begin{table}[h!]
    \centering
    \renewcommand{\arraystretch}{1.3}
    \begin{tabular}{c|c|c}
     Coefficients & Powers & Condition \\
    \hline
        $\lambda_1, \lambda_2 > 0$ & $0< \alpha_1 < \alpha_2 < \alpha_3$ & $E < 0$ \\
        $\lambda_1, \lambda_2 < 0$ & $4< \alpha_1 < \alpha_2 < \alpha_3$& $E < 0$ \\
        $\lambda_1, \lambda_2 < 0$ & $\alpha_1 < 4 < \alpha_2 < \alpha_3$& $ E < -\frac{|\lambda_1|C(\delta)(\alpha_2-\alpha_1)}{\alpha_2(\alpha_1 + 2)}M$ \\
        $\lambda_1, \lambda_2 < 0$ & $\alpha_1 < \alpha_2 < 4 < \alpha_3$ & $ E  < -\frac{1}{\alpha_3 \theta_1}\left[\frac{|\lambda_1|C(\delta_1)(\alpha_3\theta_1 - \alpha_1)}{\alpha_1 + 2}+\frac{|\lambda_2|C(\delta_2)(\alpha_3\theta_1 - \alpha_2)}{\alpha_2 + 2} \right]M $\\
        $\lambda_1 > 0, \lambda_2 < 0$ & $4 < \alpha_1 < \alpha_2 < \alpha_3$& $E < 0$ \\
        $\lambda_1 > 0, \lambda_2 < 0$ & $\alpha_1 < 4 < \alpha_2 < \alpha_3$& $E < 0$ \\
        $\lambda_1 > 0, \lambda_2 < 0$ & $\alpha_1 < \alpha_2 < 4 < \alpha_3$& $E  < -  \frac{|\lambda_2|C(\delta)(\alpha_3\theta_1-\alpha_2)}{\alpha_3\theta_1(\alpha_2 + 2)}M$\\
        $\lambda_1 < 0, \lambda_2 > 0$ & $4 < \alpha_1 < \alpha_2 < \alpha_3$& $E < -\frac{\lambda_2(\alpha_2 - \alpha_1)C(\delta)}{\alpha_1(\alpha_2 + 2)}M$ \\
        $\lambda_1 < 0, \lambda_2 > 0$ & $\alpha_1 < 4 < \alpha_2 < \alpha_3$& $E  < - \frac{|\lambda_1|((\alpha_3-\alpha_2)\theta_2 +(\alpha_2 - \alpha_1))C(\delta)}{\left[(\alpha_3-\alpha_2)\theta_2 + \alpha_2\right](\alpha_1 + 2)}M$\\
        $\lambda_1 < 0, \lambda_2 > 0$ & $\alpha_1 < \alpha_2 < 4  < \alpha_3$& $ E   < - \frac{|\lambda_1|C(\delta)(\alpha_3-\alpha_1)}{\alpha_3(\alpha_1 + 2)}M$
    \end{tabular}
\end{table}
\newpage

Then the solution $u(x,t)$ blows up in finite time, i.e., there exists a finite time $T^\ast>0$ such that
    $$\lim_{t\to T^*} \| u_x(t) \|_{L^2} = \infty.$$
\end{theorem}

Before presenting the next result, it is worth noting that certain \textit{analytic} nonlinearities such as the exponential, $\mathcal{N}(u) = e^{|u|}$, may be represented as an infinite sum of terms through a Taylor expansion. To our knowledge, this is the first result on the blow-up criteria for the NLS with an infinite sum of nonlinear terms. We use ideas from the proof of Theorem \ref{NegE_Bup_3term} to obtain the following criterion for exponential nonlinearity. 

\begin{theorem}[Negative energy blow-up for exponential nonlinearity]\label{NegE_Bup_Exp}
Let $u_0 \in H^1(\mathbb{R})$ such that $V[u_0] < \infty$ and the following condition holds 
\begin{equation}\label{E:condition-exp}
E[u_0]+ \kappa M[u_0] < 0,
\end{equation}
where $\kappa$ is defined in \eqref{E:kappa}.

Let $u(x,t)$ be a solution to \eqref{ClassicNLS} with the initial condition $u_0$  and nonlinearity
$$
\mathcal{N}(u) = \sum_{k=0}^\infty \frac{|u|^k}{k!} \equiv e^{|u|}.
$$
Then the solution $u(x,t)$ blows up in finite time. 
\end{theorem}
\begin{remark}
A more precise bound for \eqref{E:condition-exp} can be given by
    \begin{equation*}
    E[u_0] +  \frac15\left(\sum_{k=1}^4\frac{ (5-k)}{(k+2)k!}C_k + \frac{5}{2} \right)M[u_0] < 0,
    \end{equation*}
    for $C_k  = C(\delta_k)> 0$ given in \eqref{E:1}-\eqref{E:2a}.
\end{remark}

A straightforward generalization of blow-up criteria from Theorem \ref{NegE_Bup_Exp} to a finite sum or even an infinite sum of nonlinear terms in $\mathcal N(u)$ is for the case when all coefficients in the sum are non-negative (and provided the sum converges), which we state next. 

\begin{corollary}[Negative energy blow-up for finite or  infinite number of nonlinear terms]\label{Generalization}
Consider the equation \eqref{ClassicNLS} with nonlinearity $$
\mathcal{N}(u) = \sum_{k=0}^N d_k|u|^{\alpha_k},
$$ 
where $0 < N \leq \infty$ (i.e., finite or infinite number of terms).
Suppose that there exists $k^\ast \in \mathbb{N}$ such that $\alpha_{k^\ast} > 4$ and the corresponding coefficient $d_{k^\ast} > 0$ (and the rest of the coefficients $d_k \geq 0$). 

Let $u_0 \in H^1(\mathbb{R})$ be such that $V[u_0] < \infty$ and 
\begin{equation}\label{E:condition-infinite}
E[u_0]  + \left( \frac1{\alpha_{k^*}}\sum_{k=1}^{k^*-1}d_k \frac{ \alpha_{k^*}-\alpha_k }{\alpha_k + 2 }C_k \right) M[u_0] < 0,
\end{equation}
where $C_k$ is defined in \eqref{E:1Inf}-\eqref{E:2inf}. Then the solution $u(x,t)$ to the equation \eqref{ClassicNLS} with given $\mathcal N(u)$ and the initial $u_0$, blows up in finite time. 
\end{corollary}

\begin{remark}
If $\alpha_{k^*}$ is the smallest supercritical power in a given nonlinearity, then \eqref{E:condition-infinite} gives a more optimal condition for that nonlinearity. 
\end{remark}

\begin{remark}
    Corollary \ref{Generalization} is new for $N \geq 4$ as results for $N = 1 \mbox{ and } 2$ are well studied in the literature, and the case $N = 3$ follows from one of the cases in Theorem \ref{NegE_Bup_3term}.
\end{remark}

Our next result differs from the previous two theorems by establishing a criterion for {\it positive} energy blow up, in the spirit of Lushnikov \cite{Lushnikov1995}, Holmer-Roudenko-Platte \cite{HRP2010}, and Duyckaerts-Roudenko \cite{DR2015}. 
For that we recall the variance $V[u](t)$ from \eqref{E:Variance_single} and compute its first derivative $V_t(t) = 4 \mbox{Im} \int x u_x \bar{u} \, dx$. (Here we write the two combined nonlinearities with a more classic notation, denoting the powers by $p$, $q$ and coefficients by $\epsilon_1,\epsilon_2$.) 

\begin{theorem}[Positive energy blow up for two nonlinear terms]\label{DR_Bup}
Suppose that $u_0 \in H^1$ and $V[u_0] < \infty$. Assume that $\epsilon_1 < 0$, $\epsilon_2>0$ (defocusing-focusing case), $1 <  p < q$, and $q>5$. The following is a sufficient condition for blow-up in finite time for the solution $u(x,t)$ to \eqref{ClassicNLS} with positive energy $E[u_0] > 0$ and nonlinearity 
$$
\mathcal{N}(u) = \epsilon_1|u|^{p-1} + \epsilon_2|u|^{q-1}.
$$
\underline{Case $p \neq 5$:}
\begin{equation}\label{blowupcrit_two_pneq5}
   \frac{V_t(0)}{\left( \frac{|\epsilon_1|(q-p)}{(p+1)(q-1)C^*}\right)^{\frac{2}{p-1}}M^{\frac{3p+1}{2(p-1)}}} < 4\sqrt{2E^{\frac{p-5}{p-1}}}g\left( \left( \frac{(p+1)(q-1)C^*E}{|\epsilon_1|(q-p)M^{\frac{p-1}{4} + \frac{p+1}{2}}}\right)^{\frac{4}{p-1}}V[u_0] \right)
\end{equation}
where
\begin{equation}\label{DRIneqConstthm}
    C^* = \left(\frac{3p+1}{2(p+1)}\right)
    \left(\frac{\pi(3p+1)}{p-1}\right)^{\frac{p-1}{4}}\left(\frac{\Gamma(\frac{p+1}{p-1})}{\Gamma(\frac{p+1}{p-1} + \frac{1}{2})}\right)^{\frac{p-1}{2}},
\end{equation}
and

\begin{equation}
f(X)=
    \begin{cases}
        ~~\sqrt{\frac{p-1}{q-p}X^{-\frac{q-5}4} + X - \frac{q-1}{q-p}X^{\frac{5-p}{4}}} & \text{if } 0<X<1,\\
        -\sqrt{\frac{p-1}{q-p}X^{-\frac{q-5}4} + X - \frac{q-1}{q-p}X^{\frac{5-p}{4}}} & \text{if } X\geq1.
    \end{cases}
\end{equation}
\smallskip

\underline{Case $p = 5$:}
\begin{equation}\label{bup_p5}
    \frac{V_t(0)}{2M\sqrt{H}} < \sqrt{\frac{q-5}{q-1}}g\left(\frac{8(q-1)E V[u_0]}{(q-5)M^2H}\right),
\end{equation}
where $H = \left[1 + \frac{|\epsilon_1|M^2}{\pi^2}\right],$  and 

\begin{equation}
f(X)=
    \begin{cases}
        ~~\sqrt{\frac{4}{(q-5)}X^{-\frac{q-5}{4}} + X - \frac{q-1}{q-5}} & \text{if } 0<X< 1,\\
        -\sqrt{\frac{4}{(q-5)}X^{-\frac{q-5}{4}} + X - \frac{q-1}{q-5}} & \text{if } X\geq1.
    \end{cases}
\end{equation}
\end{theorem}

{
\begin{corollary}\label{Cor:RealData}
If the initial data $u_0$ are real-valued, then the conditions \eqref{blowupcrit_two_pneq5} and \eqref{bup_p5} respectively reduce to 
\begin{equation}\label{E:RV1}
V[u_0] < \left(\frac{|\epsilon_1|(q-p)M^{ \frac{3p+1}{4}}}{C^*(q-1)(p+1)E}\right)^{\frac{4}{p-1}} \qquad 
\mbox{and} \qquad 
V[u_0] < \frac{(q-5) H M^2}{ 8(q-1) E}.
\end{equation}
\end{corollary}
}

\begin{remark}
    It is important to note that when $\epsilon_1 = 0$  the inequality \eqref{bup_p5} reduces to the same criterion as in \cite{DR2015}.
\end{remark}

\begin{remark}
    We specifically consider the \textit{defocusing-focusing} case in Theorem \ref{DR_Bup}, since the criteria for the \textit{focusing-focusing} case would trivially follow by dropping the smaller nonlinear term.
\end{remark}

In Section \ref{S:Examples} we show examples that confirm our findings above: we consider several cases of initial data, which satisfy the criteria developed in Theorems \ref{NegE_Bup_3term}, \ref{NegE_Bup_Exp} and \ref{DR_Bup}. 
In particular, we consider data with negative energy and show examples of solutions to the cNLS with triple nonlinearity and to the NLS with an exponential nonlinearity in \S \ref{S:NegE_Gaussian}, then to the NLS with infinitely many terms in the nonlinearity and polynomially decaying data in \S \ref{S:NegE_Weighted}. 
The numerical simulations suggest that the solutions do indeed blow up in finite time. For the positive energy criterion, in subsection \S \ref{S:PosE_Gaussian} we consider Gaussian data (thus, exponentially decaying at infinity) and in \S \ref{S:PosE_Polynomial} polynomially decaying data (such as an inverse square) and construct examples of data with positive energy, satisfying conditions of Theorem \ref{DR_Bup}, and thus, their time evolutions blow up in finite time. These examples indicate that blow up tends to occurs for small positive energy (see Figures \ref{fig:p5q7_criteria} and \ref{fig:poly_criteria}).   
For various other numerical simulations for the NLS with combined nonlinearities (not necessarily including blow-up), we refer the reader to \cite[Section 6]{RRR2024}.
\smallskip

The remainder of the article is structured as follows: Section \ref{S:Prelim} provides an outline of useful inequalities and conserved quantities, which is referenced throughout the article. Sections \ref{S:Proof_3term} and \ref{S:Proof_exp} contain proofs for Theorems \ref{NegE_Bup_3term} and \ref{NegE_Bup_Exp}, respectively, both of which are criteria for negative energy. Section \ref{S:Proof_posEnergy} contains the proof of Theorem \ref{DR_Bup}, which is the criterion for the positive energy blow-up. Finally, Section \ref{S:Examples} provides various examples of initial data that satisfy the conditions for the blow-up criteria.
\smallskip

\section{Preliminaries}
\label{S:Prelim}
In the following section, we define conserved quantities, the virial identity, and other inequalities that are used for the equation \eqref{ClassicNLS} with nonlinearity $\mathcal{N}(u) = \sum_{k=0}^N d_k|u|^{\alpha_k}$, where $0 < N \leq \infty$. 

\subsection{Mass, energy, and the virial argument}\label{S:MEV}
We begin by stating some conserved quantities for combined nonlinear terms, mass and energy
\begin{align}\label{E:Mass}
	M[u](t) & = \int_{-\infty}^\infty|u(x,t)|^2 \, dx = M[u_0], \\
\label{E:Energy}
	E[u](t) & = \frac{1}{2}\int_{-\infty}^\infty| u_x(x,t)|^2dx -\sum_{k=0}^{N}\frac{d_k}{\alpha_k+2} \int_{-\infty}^\infty|u(x,t)|^{\alpha_k + 2} \, dx = E[u_0],
\end{align}
(at least formally in the last expression for the case of $N=\infty$, since the convergence of series would be needed as $N\to \infty$). We also utilize the variance
\begin{equation}\label{Variance}
    V[u](t) = \int_{-\infty}^\infty|xu(x,t)|^2dx,
\end{equation}
which dictates how \textit{spread out} our solution $u(x,t)$ is in space. The goal in general for the negative energy is to show that $V_{tt} < 0$ ($V$ is concave down), which would imply that for some finite time $T^*>0$, the (positive) variance becomes 0, thus leading to a contradiction. The collapse or blow-up of the solution can be thought of as it is culminating to a single \textit{blow-up} point.

To give an idea of the argument, we differentiate our variance \eqref{Variance} in time, obtaining
\begin{align}\label{Vdt}
    V_t[u](t) &= 2i\int_{-\infty}^\infty x \left[\bar{u}_xu - u_x\bar{u}\right]dx = 4\Im\int_{-\infty}^\infty xu_x\bar{u}dx =: -4y(t),
\end{align}
and
\begin{align}
    V_{tt}[u](t)= -4y^\prime(t) &= 8\int_{-\infty }^\infty|u_x|^2dx - 4\sum_{k=0}^N d_k\left(\frac{\alpha_k}{\alpha_k + 2}\right) \int_{-\infty}^\infty |u|^{\alpha_k + 2}dx \\ \label{Vdt2}
    &= 8 \| u_x \|_{L^2}^2 - 4\sum_{k=0}^N d_k\left(\frac{\alpha_k}{\alpha_k + 2}\right)\|u\|^{\alpha_k + 2}_{L^{\alpha_k+2}}.
\end{align}
Through our definition of $y(t)$, by the Cauchy-Schwarz inequality we obtain the following bound 
\begin{equation*}
    y(t) \leq \int_{-\infty}^\infty |xu_x\bar{u}|dx \leq \|u_x\|_{L^2} \|xu\|_{L^2}, 
\end{equation*}
and thus,
\begin{equation}\label{E:yineq}
    \frac{y(t)}{\|xu\|_{L^2} } \leq \|u_x\|_{L^2}.
\end{equation}
Utilizing the inequalities \eqref{E:yineq} and $y^\prime(t) \geq \epsilon \|u_x\|_{L^2}^2$ (which we show for each case in the proof of Theorem \ref{NegE_Bup_3term} below), an argument in \cite[Section 6]{TVZ2005} yields an ODE that implies the existence of a finite time $T^\ast$ :
\begin{equation}\label{E:time} 
0 < T^* <\frac{\, \,\|xu_0\|_{L^2}^2}{C y_0}
\end{equation}
such that
$$\lim_{t\to T^*} \|u_x(t) \|_{L^2} = \infty.$$

We also note that if the initial condition $u_0$ is real-valued, then the first derivative of the variance $V_t[u](0)=0$. This is useful later when we consider examples of the real-valued initial data.

We recall a simple yet crucial interpolation lemma for the negative energy estimates. This inequality is mentioned in passing in \cite[Section 6]{TVZ2005}, we state it here for convenience, outlining a proof.

\begin{lemma}\label{L:1}
    For positive constants $ a> 0$, we have for $2 < p < q$,
    \begin{equation}\label{E:Youngs}
        a^p \leq \left(\frac{q-p}{q-2}\right) a^2 + \left(\frac{p-2}{q-2}\right) a^q.
    \end{equation}
\end{lemma}
\begin{proof}
    For $2 < p < q$, we write,
    \begin{equation*}
        a^p = a^{2\theta}\cdot a^{q(1-\theta)} \;\mbox{ for } \; \theta = \frac{q-p}{q-2}. 
    \end{equation*}
    Notice that by definition $0 < \theta < 1$. Then, by concavity of the logarithm
    \begin{equation*}
        \ln(a^p) \leq \theta\ln(a^{2}) + (1-\theta)\ln(a^{q}) \leq \ln(\theta a^{2} + (1-\theta)a^{q}).
    \end{equation*}
Exponentiating both sides, we obtain our desired result.
\end{proof}

\begin{remark}\label{R:2.2} 
One may choose $\delta>0$ in Lemma \ref{L:1}, then there exists $C(\delta)>0$ such that the inequality \eqref{E:Youngs} generalizes as follows:
    \begin{equation}\label{E:Youngs_general}
        a^p \leq C(\delta) a^2 + \delta a^q,
    \end{equation}
    where $C(\delta)$ is appropriately large if $\delta$ is small (and vise versa). In fact, the sharp $C(\delta)$ is given by  
    \begin{equation}\label{E:C_delta}
       C(\delta) = \left(\frac{q-p}{q-2}\right) \left(\delta^{-1}\frac{p-2}{q-2}\right)^\frac{p-2}{q-p}.
    \end{equation} This version of the inequality is used in proving the blow-up criteria for two nonlinear terms by Tao-Visan-Zhang \cite[Section 6]{TVZ2005}.
\end{remark}

\subsection{Initial data with a quadratic phase}
For Theorems \ref{NegE_Bup_3term} and \ref{NegE_Bup_Exp}, we explicitly require the energy \eqref{E:Energy} to be negative. Simultaneously, we would also like to have some connection between blow-up and the exponent of some quadratic phase as described in the simulations of \cite[Section 6]{RRR2024}. We can ensure this is the case by choosing initial data 
\begin{equation}\label{E:InitialDataNegE}
u_0(x) = e^{\frac{ib|x|^2}{4}}v_0, \quad \mbox{with} \quad b < 0.
\end{equation}
Then
\begin{equation}\label{E:Energy_sub_negE}
    E[u_0] = E[v_0] + \frac{b^2}{8}\int_{-\infty}^\infty|x v_0(x)|^2dx + \frac{b}{2} \Im \int_{-\infty}^\infty x\,  \partial_x(v_0(x)) \, \overline{v_0(x)} \, dx.
\end{equation}
We pick initial condition $v_0$ such that $E[v_0] < 0$ and also the following inequality holds 
\begin{equation}\label{E:NegE_requirement}
    \frac{|b|}{4}\int_{-\infty}^\infty|x v_0(x)|^2dx <  \Im \int_{-\infty}^\infty x \partial_x(v_0(x))\overline{v_0(x)}dx,
\end{equation}
ensuring that the energy of $u_0$ in \eqref{E:Energy_sub_negE} is negative ($E[u_0]<0$). For examples of such initial data that satisfy this condition, see Section \ref{S:Examples} below.

\section{Blow-up for three nonlinearities}\label{S:Proof_3term}

In this section we obtain a blow-up criteria in the style of \cite{Glassey1977} and \cite{TVZ2005} for the following combined NLS with three terms 
\begin{equation}\label{3CNLS}
    \begin{cases}
        & iu_t + u_{xx} = \lambda_1 |u|^{\alpha_1}u + \lambda_2 |u|^{\alpha_2}u + \lambda_3 |u|^{\alpha_3}u \\
        &u(x,0) = u_{0} \in H^1(\mathbb R), \\
    \end{cases}
\end{equation}
where $(x,t) \in \R\times \R^+$, $\lambda_j \in \R$ and $\alpha_j \in \R^+$.

From Section \ref{S:Prelim}, we take the virial identity \eqref{Vdt2} with $(d_1, d_2, d_3,...,d_N) = (-\lambda_1, -\lambda_2,-\lambda_3,0, ...,0)$ and $(\alpha_1, \alpha_2,\alpha_3, ...,\alpha_N) = (\alpha_1, \alpha_2,\alpha_3,0,...,0)$, which becomes
\begin{equation}\label{E:Vdt2_3nonlinear}
    V_{tt} = 8 \|u_x\|^2_{L^2} + \frac{4\lambda_1\alpha_1}{\alpha_1 + 2}\|u\|^{\alpha_1 + 2}_{L^{\alpha_1 + 2}} + \frac{4\lambda_2\alpha_2}{\alpha_2 + 2}\|u\|^{\alpha_2 + 2}_{L^{\alpha_2 + 2}} + \frac{4\lambda_3\alpha_3}{\alpha_3 + 2}\|u\|^{\alpha_3 + 2}_{L^{\alpha_3 + 2}},
\end{equation}
and the energy \eqref{E:Energy} in this case is given by
\begin{equation}\label{E:Energy_3nonlinear}
    E[u] = \frac{1}{2} \|u_x\|^2_{L^2} + \frac{\lambda_1}{\alpha_1 + 2}\|u\|^{\alpha_1 + 2}_{L^{\alpha_1 + 2}} + \frac{\lambda_2}{\alpha_2 + 2}\|u\|^{\alpha_2 + 2}_{L^{\alpha_2 + 2}} + \frac{\lambda_3}{\alpha_3 + 2}\|u\|^{\alpha_3 + 2}_{L^{\alpha_3 + 2}}.
\end{equation}

For the rest of the proof we recall that $\alpha_3>4$, $0 < \alpha_1 < \alpha_2 < \alpha_3$ and $\lambda_3 < 0$.

\subsection{Proof of Theorem \ref{NegE_Bup_3term}}
Our goal is to show the finite time blow-up through a convexity argument. Thus, we aim to show that $V_{tt}(t) \leq 0$ at some time $t = T^*$. Our approach takes ideas from \cite{TVZ2005} with delicate splittings, interpolations and appropriate norm bounds. We begin as in \eqref{Vdt2} by setting, 
$V_{tt} = -4y^\prime(t),$
where 
\begin{equation}\label{yprime}
    y^\prime(t) = -2\|u_x\|^2_{L^2} - \frac{\lambda_1\alpha_1}{\alpha_1 + 2}\|u\|^{\alpha_1 + 2}_{L^{\alpha_1 + 2}} - \frac{\lambda_2\alpha_2}{\alpha_2 + 2}\|u\|^{\alpha_2 + 2}_{L^{\alpha_2 + 2}} - \frac{\lambda_3\alpha_3}{\alpha_3 + 2}\|u\|^{\alpha_3 + 2}_{L^{\alpha_3 + 2}}.
\end{equation}

Then it is sufficient to show that $y^\prime(t) > 0$. Notice, however, that the signs of the coefficients and the magnitude of the exponents influence which direction we take throughout the proof, therefore, we split the argument into the cases as shown in Table \ref{tab:Bup_cases}.  
{\small 
\begin{table}[h!]
    \centering
    \begin{tabular}{c|c|c}
    Case Number & Coefficients & Powers ($\alpha_3 > 4$)\\
    \hline
        Case 1~ & $\lambda_1, \lambda_2 > 0$ & $0< \alpha_1 < \alpha_2 < \alpha_3$\\
        Case 2A & $\lambda_1, \lambda_2 < 0$ & $4< \alpha_1 < \alpha_2 < \alpha_3$\\
        Case 2B & $\lambda_1, \lambda_2 < 0$ & $\alpha_1 < 4 < \alpha_2 < \alpha_3$\\
        Case 2C & $\lambda_1, \lambda_2 < 0$ & $\alpha_1 < \alpha_2 < 4 < \alpha_3$ \\
        Case 3A & $\lambda_1 > 0, \lambda_2 < 0$ & $4 < \alpha_1 < \alpha_2 < \alpha_3$\\
        Case 3B & $\lambda_1 > 0, \lambda_2 < 0$ & $\alpha_1 < 4 < \alpha_2 < \alpha_3$\\
        Case 3C & $\lambda_1 > 0, \lambda_2 < 0$ & $\alpha_1 < \alpha_2 < 4 < \alpha_3$\\
        Case 4A & $\lambda_1 < 0, \lambda_2 > 0$ & $4 < \alpha_1 < \alpha_2 < \alpha_3$\\
        Case 4B & $\lambda_1 < 0, \lambda_2 > 0$ & $\alpha_1 < 4 < \alpha_2 < \alpha_3$\\
        Case 4C & $\lambda_1 < 0, \lambda_2 > 0$ & $\alpha_1 < \alpha_2 < 4  < \alpha_3$
     \end{tabular}
    \caption{\small Cases of coefficients and powers for three nonlinear terms in the proof of Theorem \ref{NegE_Bup_3term}.}
    \label{tab:Bup_cases}
\end{table}}

For each case we start by writing \eqref{yprime}, solving for the $\|u\|_{L^{\alpha_j+2}}^{\alpha_j+2}$ term in \eqref{E:Energy_3nonlinear} (which will depend on each case),  and replacing the associated term in \eqref{yprime}, from there the cases differ. In some cases the computation is straightforward, but in others we apply some interpolation estimates. When a new technique is introduced, it will be described; otherwise, the computation will be brief.
\smallskip

\textbf{Case 1: } ($\lambda_1, \lambda_2 > 0$, $0< \alpha_1 < \alpha_2 < \alpha_3$ and $E <0$).

Starting from \eqref{yprime}, notice that in this case, the terms that correspond to $\alpha_1$ and $\alpha_2$ and $u_x$ have negative coefficient, to correct for this, we make the substitution (using \eqref{E:Energy_3nonlinear}),

\begin{equation}\label{a3_energy}
    - \frac{\lambda_3\alpha_3}{\alpha_3 + 2}\|u\|^{\alpha_3 + 2}_{L^{\alpha_3 + 2}} = \alpha_3\left[\frac{1}{2} \|u_x\|^2_{L^2} + \frac{\lambda_1}{\alpha_1 + 2}\|u\|^{\alpha_1 + 2}_{L^{\alpha_1 + 2}} + \frac{\lambda_2}{\alpha_2 + 2}\|u\|^{\alpha_2 + 2}_{L^{\alpha_2 + 2}} - E\right].
\end{equation}
Then,
\begin{align*}
    y^\prime(t) &= -2\|u_x\|^2_{L^2} - \frac{\lambda_1\alpha_1}{\alpha_1 + 2}\|u\|^{\alpha_1 + 2}_{L^{\alpha_1 + 2}} - \frac{\lambda_2\alpha_2}{\alpha_2 + 2}\|u\|^{\alpha_2 + 2}_{L^{\alpha_2 + 2}}+ \\ &\hspace{2.3in}+ \alpha_3\left[\frac{1}{2} \|u_x\|^2_{L^2} + \frac{\lambda_1}{\alpha_1 + 2}\|u\|^{\alpha_1 + 2}_{L^{\alpha_1 + 2}} + \frac{\lambda_2}{\alpha_2 + 2}\|u\|^{\alpha_2 + 2}_{L^{\alpha_2 + 2}} - E\right] \\
    &= \left[\frac{\alpha_3}{2} - 2\right]\|u_x\|^2_2 + \frac{\lambda_1(\alpha_3 - \alpha_1)}{\alpha_1 + 2}\|u\|^{\alpha_1 + 2}_{L^{\alpha_1 + 2}} + \frac{\lambda_2(\alpha_3-\alpha_2)}{\alpha_2 + 2}\|u\|^{\alpha_2 + 2}_{L^{\alpha_2 + 2}} - E\alpha_3 \\&\geq \frac{\alpha_3 -4}{2}\|u_x\|^2_{L^2} - E\alpha_3 > 0,
\end{align*}
provided that $E<0$ (and $\alpha_3>4$).
\smallskip

\textbf{Case 2A: } ($\lambda_1, \lambda_2 < 0$, $4< \alpha_1 < \alpha_2 < \alpha_3$ and $E <0$)

In this case, all nonlinear terms are focusing, and thus, only the term corresponding to $u_x$ has negative coefficient. Thus, we use a substitution similar to \eqref{a3_energy} (but with $\alpha_1$), reducing this case to the same result and condition on the energy as in \textbf{Case 1}.
\smallskip

\textbf{Case 2B: } ($\lambda_1, \lambda_2 < 0$, $ \alpha_1 < 4 < \alpha_2 < \alpha_3$ and $E <0$). In this case, $\alpha_1 < 4$, thus, we cannot proceed as in case 2A. Instead, we apply \eqref{E:Youngs_general} to the $\alpha_2$ term and then proceed as usual,
{\small
\begin{align*}
    y^\prime(t) &= -2\|u_x\|^2_{L^2} - \frac{\lambda_1\alpha_1}{\alpha_1 + 2}\|u\|^{\alpha_1 + 2}_{L^{\alpha_1 + 2}}  - \frac{\lambda_2\alpha_2}{\alpha_2 + 2}\|u\|^{\alpha_2 + 2}_{L^{\alpha_2 + 2}} - \frac{\lambda_3\alpha_3}{\alpha_3 + 2}\|u\|^{\alpha_3 + 2}_{L^{\alpha_3 + 2}} \\
    &= -2\|u_x\|^2_{L^2} - \frac{\lambda_1\alpha_1}{\alpha_1 + 2}\|u\|^{\alpha_1 + 2}_{L^{\alpha_1 + 2}}  + \alpha_2 \left[\frac{1}{2} \|u_x\|^2_2 + \frac{\lambda_1}{\alpha_1 + 2}\|u\|^{\alpha_1 + 2}_{L^{\alpha_1 + 2}} + \frac{\lambda_3}{\alpha_3 + 2}\|u\|^{\alpha_3 + 2}_{L^{\alpha_3 + 2}} - E\right] \\ &\hspace{4in}- 
     \frac{\lambda_3\alpha_3}{\alpha_3 + 2}\|u\|^{\alpha_3 + 2}_{L^{\alpha_3 + 2}} \\
    & = \left[\frac{\alpha_2- 4}{2}\right]\|u_x\|^2_{L^2} + \underbrace{\underbrace{\frac{\lambda_1(\alpha_2-\alpha_1)}{\alpha_1 + 2}}_{A}\|u\|^{\alpha_1 + 2}_{L^{\alpha_1 + 2}}}_{*} + \underbrace{\frac{\lambda_3(\alpha_2-\alpha_3)}{\alpha_3 + 2}\|u\|^{\alpha_3 + 2}_{L^{\alpha_3 + 2}}}_{\mathclap{\substack{\text{positive} \\ \text{($\lambda_3 <0 $ and $\alpha_2-\alpha_3 < 0$)}}}} - E\alpha_2.
\end{align*}}
For any $\delta > 0$, we can apply 

\eqref{E:Youngs_general} to $|*|$ to get 
$$
|*| \leq |A| \big(C(\delta)\|u\|_{L^2}^{2} + \delta\|u\|^{\alpha_3+2}_{L^{\alpha_3+2}} \big),
$$
and thus, we choose $\delta=\delta_{2B}$ such that the following condition holds

\begin{equation}\label{E:delta-2B}
\delta_{2B} \leq \frac{|\lambda_3|}{|\lambda_1|}\frac{|\alpha_2 - \alpha_3|}{(\alpha_2 - \alpha_1)}\frac{(\alpha_1+2)}{(\alpha_3 + 2)}. 
\end{equation}

This leads to the lower bound
\begin{align*}
    y^\prime(t) &\geq \left[\frac{\alpha_2 - 4}{2}\right]\|u_x\|^2_{L^2} + \frac{\lambda_1(\alpha_2-\alpha_1)}{\alpha_1 + 2}(C(\delta_{2B})\|u\|_{L^2}^{2} + \delta_{2B}\|u\|^{\alpha_3+2}_{L^{\alpha_3+2}}) +\frac{\lambda_3(\alpha_2-\alpha_3)}{\alpha_3 + 2}\|u\|^{\alpha_3 + 2}_{L^{\alpha_2 + 2}} - E\alpha_2 \\
    &\geq \left[\frac{\alpha_2 - 4}{2}\right]\|u_x\|^2_{L^2} + \frac{\lambda_1C(\delta_{2B})(\alpha_2-\alpha_1)}{\alpha_1 + 2}M - E\alpha_2 >0,
\end{align*}
provided $E < -\frac{|\lambda_1|C(\delta_{2B})(\alpha_2-\alpha_1)}{\alpha_2(\alpha_1 + 2)}M$ and $\alpha_2 >4$ by the assumption in this case. 
\smallskip

\textbf{Case 2C: } ($\lambda_1, \lambda_2 < 0$, $ \alpha_1 < \alpha_2 < 4 < \alpha_3$ and $E <0$). As in \textbf{Case 2B}, we cannot substitute the $\alpha_1$ (and $\alpha_2$) term since it does not resolve the negative coefficient for the $u_x$ term. Directly substituting for $\alpha_3$ is also not enough as the coefficients $\lambda_1$ and $\lambda_2$ are negative. Thus, we split the $\alpha_3$ into parts and utilize \eqref{E:Youngs_general}, for any $\theta \in (0,1)$, thus, obtaining
{\small 
\begin{align*}
    y^\prime(t) &= -2\|u_x\|^2_{L^2} - \frac{\lambda_1\alpha_1}{\alpha_1 + 2}\|u\|^{\alpha_1 + 2}_{L^{\alpha_1 + 2}} - \frac{\lambda_2\alpha_2}{\alpha_2 + 2}\|u\|^{\alpha_2 + 2}_{L^{\alpha_2 + 2}} - \frac{\lambda_3\alpha_3\theta}{\alpha_3 + 2}\|u\|^{\alpha_3 + 2}_{L^{\alpha_3 + 2}} - \frac{\lambda_3\alpha_3(1-\theta)}{\alpha_3 + 2}\|u\|^{\alpha_3 + 2}_{L^{\alpha_3 + 2}} \\
    & = -2\|u_x\|^2_{L^2} - \frac{\lambda_1\alpha_1}{\alpha_1 + 2}\|u\|^{\alpha_1 + 2}_{L^{\alpha_1 + 2}} - \frac{\lambda_2\alpha_2}{\alpha_2 + 2}\|u\|^{\alpha_2 + 2}_{L^{\alpha_2 + 2}} + \\
    &+ \alpha_3\theta \left[\frac{1}{2} \|u_x\|^2_{L^2} + \frac{\lambda_1}{\alpha_1 + 2}\|u\|^{\alpha_1 + 2}_{L^{\alpha_1 + 2}} + \frac{\lambda_2}{\alpha_2 + 2}\|u\|^{\alpha_2 + 2}_{L^{\alpha_2 + 2}} - E\right] - \frac{\lambda_3\alpha_3(1-\theta)}{\alpha_3 + 2}\|u\|^{\alpha_3 + 2}_{L^{\alpha_3 + 2}}\\
    &= \left[\frac{\alpha_3\theta}{2} - 2\right]\|u_x\|^2_2 + \underbrace{\underbrace{\frac{\lambda_1(\alpha_3\theta - \alpha_1)}{\alpha_1 + 2}}_{A}\|u\|^{\alpha_1 + 2}_{L^{\alpha_1 + 2}}}_{*} + \underbrace{\underbrace{\frac{\lambda_2(\alpha_3\theta - \alpha_2)}{\alpha_2 + 2}}_{B}\|u\|^{\alpha_2 + 2}_{L^{\alpha_2 + 2}}}_{**} - \frac{\lambda_3\alpha_3(1-\theta)}{\alpha_3 + 2}\|u\|^{\alpha_3 + 2}_{L^{\alpha_3 + 2}} - E\alpha_3\theta.
\end{align*}
}
Note that to keep a positive sign of the first term, we require $\theta > \frac{4}{\alpha_3}$. Then for any $\delta_1,\delta_2>0$ an application of \eqref{E:Youngs_general} on $|*|$ and $|**|$ yields
\begin{equation}\label{E:delta-2C-b}
|*| \leq |A| \big(C(\delta_1)\|u\|_{L^2}^2 + \delta_1\|u\|^{\alpha_3+2}_{L^{\alpha_3+2}} \big) 
\quad \mbox{and} \quad
|**| \leq |B|\big( C(\delta_2)\|u\|_{L^2}^2 + \delta_2\|u\|^{\alpha_3+2}_{L^{\alpha_3+2}} \big).
\end{equation}
Choose $\delta_1 = \delta_{2C,1}>0$ and $\delta_2 = \delta_{2C,2}>0$ such that
\begin{equation}\label{E:delta-2C-c}
\delta_{2C,1}|A| \leq \frac{|\lambda_3|\alpha_3(1-\theta)}{2(\alpha_3 + 2)} \quad \mbox{and} \quad
\delta_{2C,2}|B| \leq \frac{|\lambda_3|\alpha_3(1-\theta)}{2(\alpha_3 + 2)}.
\end{equation}
Simplifying, we obtain the following condition: fix $\theta \in (\frac{4}{\alpha_3},1)$, then choose $\delta_{2C,j}$, $j=1,2$, such that
\begin{equation}\label{E:delta-2C}
0< \delta_{2C,~j} \leq  \frac{~~|\lambda_3| (\alpha_j+2)}{2|\lambda_j| (\alpha_3+2)}  \frac{(1-\theta)}{(\theta - {\alpha_j}/{\alpha_3})}, ~~j=1,2.
\end{equation}
Continuing with the lower bound, we obtain
\begin{align*}
    y^\prime(t) &\geq \frac{\alpha_3\theta - 4}{2} \|u_x\|^2_{L^2} + \frac{\lambda_1(\alpha_3\theta - \alpha_1)}{\alpha_1 + 2}(C(\delta_{2C,1})\|u\|_{L^2}^2 + \delta_{2C,1}\|u\|^{\alpha_3+2}_{L^{\alpha_3+2}}) \\ 
    &\qquad+\frac{\lambda_2(\alpha_3\theta - \alpha_2)}{\alpha_2 + 2}(C(\delta_{2C,2})\|u\|_{L^2}^2 + \delta_{2C,2}\|u\|^{\alpha_3+2}_{L^{\alpha_3+2}}) - \frac{\lambda_3\alpha_3(1-\theta)}{\alpha_3 + 2}\|u\|^{\alpha_3 + 2}_{L^{\alpha_3 + 2}} - E\alpha_3\theta \\
    &\geq \frac{\alpha_3\theta - 4}{2}\|u_x\|^2_{L^2} - E\theta \alpha_3 + \left[\frac{\lambda_1C(\delta_{2C,1})(\alpha_3\theta - \alpha_1)}{\alpha_1 + 2}+\frac{\lambda_2C(\delta_{2C,2})(\alpha_3\theta - \alpha_2)}{\alpha_2 + 2} \right]M > 0,
\end{align*}
provided $E < -\frac{1}{\alpha_3 \theta}\left[\frac{|\lambda_1|C(\delta_{2C,1})(\alpha_3\theta - \alpha_1)}{\alpha_1 + 2}+\frac{|\lambda_2|C(\delta_{2C,2})(\alpha_3\theta - \alpha_2)}{\alpha_2 + 2} \right]M $ and $\theta >4/\alpha_3$ as assumed above.

\textbf{Case 3A: } ($\lambda_1 > 0, \lambda_2 < 0$, $  4 < \alpha_1 < \alpha_2 < \alpha_3$ and $E <0$).

A substitution similar to \eqref{a3_energy} but for $\alpha_1$ reduces this case to similar calculation and result as \textbf{Cases 1} and \textbf{2A}.
\smallskip

\textbf{Case 3B: } ($\lambda_1 > 0, \lambda_2 < 0$, $ \alpha_1 < 4 < \alpha_2 < \alpha_3$ and $E <0$). 

A substitution of \eqref{a3_energy} but for $\alpha_2$ reduces this case to a similar calculation and result as in \textbf{Cases 1}, \textbf{2A}, and \textbf{3A}.
\smallskip

\textbf{Case 3C: } ($\lambda_1 > 0, \lambda_2 < 0$, $  \alpha_1 < \alpha_2 < 4 < \alpha_3$ and $E <0$). This case is treated similarly to that of \textbf{Case 2B} and \textbf{Case 2C}, since $\lambda_1 > 0$, and thus, that term is easily bounded away, yielding
\begin{align*}
    y^\prime(t) &= -2\|u_x\|^2_{L^2} - \frac{\lambda_1\alpha_1}{\alpha_1 + 2}\|u\|^{\alpha_1 + 2}_{L^{\alpha_1 + 2}} - \frac{\lambda_2\alpha_2}{\alpha_2 + 2}\|u\|^{\alpha_2 + 2}_{L^{\alpha_2 + 2}} - \frac{\lambda_3\alpha_3\theta}{\alpha_3 + 2}\|u\|^{\alpha_3 + 2}_{L^{\alpha_3 + 2}} - \frac{\lambda_3\alpha_3(1-\theta)}{\alpha_3 + 2}\|u\|^{\alpha_3 + 2}_{L^{\alpha_3 + 2}} \\
    & = -2\|u_x\|^2_2 - \frac{\lambda_1\alpha_1}{\alpha_1 + 2}\|u\|^{\alpha_1 + 2}_{L^{\alpha_1 + 2}} - \frac{\lambda_2\alpha_2}{\alpha_2 + 2}\|u\|^{\alpha_2 + 2}_{L^{\alpha_2 + 2}} + \\
    &+ \alpha_3\theta \left[\frac{1}{2} \|u_x\|^2_{L^2} + \frac{\lambda_1}{\alpha_1 + 2}\|u\|^{\alpha_1 + 2}_{L^{\alpha_1 + 2}} + \frac{\lambda_2}{\alpha_2 + 2}\|u\|^{\alpha_2 + 2}_{L^{\alpha_2 + 2}} - E\right] - \frac{\lambda_3\alpha_3(1-\theta)}{\alpha_3 + 2}\|u\|^{\alpha_3 + 2}_{L^{\alpha_3 + 2}}\\
    &= \left[\frac{\alpha_3\theta}{2} - 2\right]\|u_x\|^2_2 + \underbrace{\frac{\lambda_1(\alpha_3\theta - \alpha_1)}{\alpha_1 + 2}\|u\|^{\alpha_1 + 2}_{L^{\alpha_1 + 2}}}_{positive} + \underbrace{\underbrace{\frac{\lambda_2(\alpha_3\theta - \alpha_2)}{\alpha_2 + 2}}_{B}\|u\|^{\alpha_2 + 2}_{L^{\alpha_2 + 2}}}_{*} - \frac{\lambda_3\alpha_3(1-\theta)}{\alpha_3 + 2}\|u\|^{\alpha_3 + 2}_{L^{\alpha_3 + 2}} - E\alpha_3\theta.
\end{align*}
We again require $\theta > \frac{4}{\alpha_3}$ to keep the first term positive. Then for any $\delta > 0$, we apply \eqref{E:Youngs_general} to $|*|$ to obtain
\begin{equation}
    |*| \leq |B|(C(\delta)\|u\|_{L^2}^2 + \delta\|u\|^{\alpha_3+2}_{L^{\alpha_3+2}}),
\end{equation}
and thus, we choose $\delta=\delta_{3C}>0$ such that the following condition holds 
\begin{equation}\label{E:delta-3C}
\delta_{3C} \leq \frac{|\lambda_3|}{|\lambda_2|}\frac{(\alpha_2 + 2)}{(\alpha_3 + 2)}\frac{(1-\theta)}{(\theta - \alpha_2/\alpha_3)}. 
\end{equation}

Thus,
\begin{align*}
    y^\prime(t) &\geq \left[\frac{\alpha_3\theta}{2} - 2\right]\|u_x\|^2_{L^2} + \frac{\lambda_2(\alpha_3\theta-\alpha_2)}{\alpha_2 + 2}(C(\delta_{3C})\|u\|_{L^2}^2 + \delta_{3C}\|u\|^{\alpha_3+2}_{L^{\alpha_3+2}}) -\frac{\lambda_3\alpha_3(1-\theta)}{\alpha_3 + 2}\|u\|^{\alpha_3 + 2}_{L^{\alpha_3 + 2}} - E\alpha_3\theta \\
    &\geq \left[\frac{\alpha_3\theta}{2} - 2\right]\|u_x\|^2_{L^2} + \frac{\lambda_2C(\delta_{3C})(\alpha_3\theta-\alpha_2)}{\alpha_2 + 2}M - E\alpha_3\theta>0,
\end{align*}
provided $E < -  \frac{|\lambda_2|C(\delta_{3C})(\alpha_3\theta-\alpha_2)}{\alpha_3\theta(\alpha_2 + 2)}M$ and $\theta >4/\alpha_3$ as assumed above. 
\begin{remark}
    Note that in this case, $\theta$ is specifically required, as we need a residual positive term to cancel out with \eqref{E:delta-3C} and $L^{\alpha_3 + 2}$ term. Unlike other cases, we cannot use the \textit{positive} $L^{\alpha_1 + 2}$ term since $\alpha_1 < \alpha_2$. 
\end{remark}
\textbf{Case 4A: } ($\lambda_1 < 0, \lambda_2 > 0$, $ 4 < \alpha_1 < \alpha_2 < \alpha_3$ and $E <0$). This case is treated similar to that of \textbf{Case 1}, but since $\lambda_2 > 0$ it cannot simply be bounded away, so we use it to apply \eqref{E:Youngs_general},
{\small\begin{align*}
    y^\prime(t) &= -2\|u_x\|^2_{L^2} + \alpha_1\left[\frac{1}{2} \|u_x\|^2_{L^2} + \frac{\lambda_2}{\alpha_2 + 2}\|u\|^{\alpha_2 + 2}_{L^{\alpha_2 + 2}} + \frac{\lambda_3}{\alpha_3 + 2}\|u\|^{\alpha_3 + 2}_{L^{\alpha_3 + 2}} - E\right] - \frac{\lambda_2\alpha_2}{\alpha_2 + 2}\|u\|^{\alpha_2 + 2}_{L^{\alpha_2 + 2}} - \frac{\lambda_3\alpha_3}{\alpha_3 + 2}\|u\|^{\alpha_3 + 2}_{L^{\alpha_3 + 2}} \\
    &= \left[\frac{\alpha_1}{2} - 2\right]\|u_x\|^2_{L^2} \underbrace{\underbrace{-\frac{\lambda_2(\alpha_2 - \alpha_1)}{\alpha_2 + 2}}_{A}\|u\|^{\alpha_2 + 2}_{L^{\alpha_2 + 2}}}_{*} + \underbrace{\frac{\lambda_3(\alpha_1-\alpha_3)}{\alpha_3 + 2}\|u\|^{\alpha_3 + 2}_{L^{\alpha_3 + 2}}}_{\text{positive}} - E\alpha_1.
\end{align*}}
Then for any $\delta > 0$, we can apply 
\eqref{E:Youngs_general} to $|*|$ to get 
\begin{equation}
    |*| \leq |A|(C(\delta)\|u\|_{L^2}^2 + \delta\|u\|^{\alpha_3+2}_{L^{\alpha_3+2}}),
\end{equation}
and thus, we choose $\delta=\delta_{4A}>0$ such that the following condition holds 
\begin{equation}\label{E:delta-4A}
\delta_{4A} \leq \frac{|\lambda_3|}{|\lambda_2|}\frac{|\alpha_1 - \alpha_3|}{(\alpha_2 - \alpha_1)} \frac{(\alpha_2 +2)}{(\alpha_3 + 2)}. 
\end{equation}
Thus,
\begin{align*}
    y^\prime(t) &\geq \left[\frac{\alpha_1}{2} - 2\right]\|u_x\|^2_{L^2} - \left[ E\alpha_1 + \frac{\lambda_2(\alpha_2 - \alpha_1)C(\delta_{4A})}{\alpha_2 + 2}M \right]>0, 
\end{align*}
provided $E < -\frac{\lambda_2(\alpha_2 - \alpha_1)C(\delta_{4A})}{\alpha_1(\alpha_2 + 2)}M$ and $\alpha_1 >4$ by the assumption in this case.

\textbf{Case 4B: } ($\lambda_1 < 0, \lambda_2 > 0$, $  \alpha_1 < 4 < \alpha_2 < \alpha_3$ and $E <0$). This case closely resembles \textbf{Case 2B}. However, the sign of $\lambda_2$ is flipped, which causes a sign mismatch when splitting and using Young's inequality. To avoid this, we apply the substitution \eqref{a3_energy}, then split the $\|u\|_{L^{\alpha_3 + 2}}^{\alpha_3 + 2}$ term into two pieces. Applying \eqref{a3_energy} once again allows us to use the $\lambda_1$ coefficient for Young's inequality and works around the sign mismatch,

{\small\begin{align*}
    y^\prime(t) &= -2\|u_x\|^2_{L^2} - \frac{\lambda_1\alpha_1}{\alpha_1 + 2}\|u\|^{\alpha_1 + 2}_{L^{\alpha_1 + 2}}  + \alpha_2\left[\frac{1}{2} \|u_x\|^2_{L^2} + \frac{\lambda_1}{\alpha_1 + 2}\|u\|^{\alpha_1 + 2}_{L^{\alpha_1 + 2}} + \frac{\lambda_3}{\alpha_3 + 2}\|u\|^{\alpha_3 + 2}_{L^{\alpha_3 + 2}} - E\right] - \frac{\lambda_3\alpha_2}{\alpha_3 + 2}\|u\|^{\alpha_3 + 2}_{L^{\alpha_3 + 2}} \\
     &= \left[\frac{\alpha_2}{2} - 2\right]\|u_x\|^2_{L^2} + \frac{\lambda_1(\alpha_2 - \alpha_1)}{\alpha_1 + 2}\|u\|^{\alpha_1 + 2}_{L^{\alpha_1 + 2}} - \frac{\lambda_3(\alpha_3-\alpha_2)}{\alpha_3 + 2}\|u\|^{\alpha_3 + 2}_{L^{\alpha_3 + 2}} - E\alpha_2 \\
    &= \left[\frac{\alpha_2}{2} - 2\right]\|u_x\|^2_{L^2} + \frac{\lambda_1(\alpha_2 - \alpha_1)}{\alpha_1 + 2}\|u\|^{\alpha_1 + 2}_{L^{\alpha_1 + 2}} - \frac{\lambda_3(\alpha_3-\alpha_2)\theta}{\alpha_3 + 2}\|u\|^{\alpha_3 + 2}_{L^{\alpha_3 + 2}} - \frac{\lambda_3(\alpha_3-\alpha_2)(1-\theta)}{\alpha_3 + 2}\|u\|^{\alpha_3 + 2}_{L^{\alpha_3 + 2}}- E\alpha_2 \\
    &= \left[\frac{\alpha_2}{2} - 2\right]\|u_x\|^2_{L^2} + \frac{\lambda_1(\alpha_2 - \alpha_1)}{\alpha_1 + 2}\|u\|^{\alpha_1 + 2}_{L^{\alpha_1 + 2}} + (\alpha_3-\alpha_2)\theta\left[\frac{1}{2} \|u_x\|^2_{L^2} + \frac{\lambda_1}{\alpha_1 + 2}\|u\|^{\alpha_1 + 2}_{L^{\alpha_1 + 2}} + \frac{\lambda_2}{\alpha_2 + 2}\|u\|^{\alpha_2 + 2}_{L^{\alpha_2 + 2}} - E\right] \\
    &\hspace{1.5in}- \frac{\lambda_3(\alpha_3-\alpha_2)(1-\theta)}{\alpha_3 + 2}\|u\|^{\alpha_3 + 2}_{L^{\alpha_3 + 2}}- E\alpha_2 \\
    &= \left[\frac{(\alpha_3 - \alpha_2)\theta + \alpha_2 -4}{2}\right]\|u_x\|^2_{L^2} + \underbrace{\underbrace{\frac{\lambda_1((\alpha_3-\alpha_2)\theta +(\alpha_2 - \alpha_1))}{\alpha_1 + 2}}_{A}\|u\|^{\alpha_1 + 2}_{L^{\alpha_1 + 2}}}_{*} \\
    &\hspace{1.5in}+ \underbrace{\frac{\lambda_2(\alpha_3-\alpha_2)\theta}{\alpha_2 + 2}\|u\|^{\alpha_2 + 2}_{L^{\alpha_2 + 2}}}_{\text{positive}} - \frac{\lambda_3(\alpha_3-\alpha_2)(1-\theta)}{\alpha_3 + 2}\|u\|^{\alpha_3 + 2}_{L^{\alpha_3 + 2}}- E\left[(\alpha_3-\alpha_2)\theta + \alpha_2\right].
\end{align*}}
Here, the coefficient in front of the first term is positive for any $\theta \in (0,1)$, since in this case $\alpha_2>4$ and also $\alpha_3>\alpha_2$. Therefore, for any $\delta >0$ applying \eqref{E:Youngs_general} to $|*|$, we get
\begin{equation}
    |*| \leq |A|(C(\delta)\|u\|_{L^2}^2 + \delta\|u\|^{\alpha_3+2}_{L^{\alpha_3+2}}).
\end{equation}
Thus, fix $\theta \in (0,1)$ and choose $\delta=\delta_{4B}$ such that the following condition holds 
\begin{equation}\label{E:delta-4B}
\delta_{4B} \leq \frac{|\lambda_3|(\alpha_1+2)  }{|\lambda_1|(\alpha_3 + 2)} \frac{(1-\theta)}{ (\theta + (\alpha_2-\alpha_1)/(\alpha_3-\alpha_2) )}. 
\end{equation}
Then we deduce the following lower bound
\begin{align*}
    y^\prime(t) & \geq \left[\frac{(\alpha_3 - \alpha_2)\theta + \alpha_2 -4}{2}\right]\|u_x\|^2_2 + \frac{\lambda_1((\alpha_3-\alpha_2)\theta +(\alpha_2 - \alpha_1))C(\delta_{4B})}{\alpha_1 + 2}M- E\left[(\alpha_3-\alpha_2)\theta + \alpha_2\right] >0,
\end{align*}
provided $E < - \frac{|\lambda_1|((\alpha_3-\alpha_2)\theta +(\alpha_2 - \alpha_1))C(\delta_{4B})}{\left[(\alpha_3-\alpha_2)\theta + \alpha_2\right](\alpha_1 + 2)}M$ and $\alpha_2 >4$ by the assumption in this case.
\smallskip

\textbf{Case 4C: } ($\lambda_1 < 0, \lambda_2 > 0$, $  \alpha_1 < \alpha_2 < 4  < \alpha_3$ and $E <0$). This case is identical to \textbf{Case 3C} but with the signs of $\lambda_1$ and $\lambda_2$ flipped, thus, we instead apply \eqref{E:Youngs_general} to {$\|u\|_{\alpha_1+2}^{\alpha_1+ 2}$}, obtaining

{\begin{align*}
    y^\prime(t) &= -2\|u_x\|^2_{L^2} - \frac{\lambda_1\alpha_1}{\alpha_1 + 2}\|u\|^{\alpha_1 + 2}_{L^{\alpha_1 + 2}} - \frac{\lambda_2\alpha_2}{\alpha_2 + 2}\|u\|^{\alpha_2 + 2}_{L^{\alpha_2 + 2}} - \frac{\lambda_3\alpha_3}{\alpha_3 + 2}\|u\|^{\alpha_3 + 2}_{L^{\alpha_3 + 2}} \\
    & \geq -2\|u_x\|^2_{L^2} - \frac{\lambda_1\alpha_1}{\alpha_1 + 2}\|u\|^{\alpha_1 + 2}_{L^{\alpha_1 + 2}} - \frac{\lambda_2\alpha_2}{\alpha_2 + 2}\|u\|^{\alpha_2 + 2}_{L^{\alpha_2 + 2}}  \\
    &\hspace{1in}+ \alpha_3 \left[\frac{1}{2} \|u_x\|^2_2 + \frac{\lambda_1}{\alpha_1 + 2}\|u\|^{\alpha_1 + 2}_{L^{\alpha_1 + 2}} + \frac{\lambda_2}{\alpha_2 + 2}\|u\|^{\alpha_2 + 2}_{L^{\alpha_2 + 2}} - E\right]\\
    &= \left[\frac{\alpha_3}{2} - 2\right]\|u_x\|^2_{L^2} + \underbrace{\underbrace{\frac{\lambda_1(\alpha_3 - \alpha_1)}{\alpha_1 + 2}}_{A}\|u\|^{\alpha_1 + 2}_{L^{\alpha_1 + 2}}}_{*} + \underbrace{\frac{\lambda_2(\alpha_3 - \alpha_2)}{\alpha_2 + 2}\|u\|^{\alpha_2 + 2}_{L^{\alpha_2 + 2}}}_{\text{positive}}  - E\alpha_3.
\end{align*}
}
Then for any $\delta > 0$, we can apply 
\eqref{E:Youngs_general} to $|*|$ to get 
\begin{equation}
    |*| \leq |A|(C(\delta)\|u\|_{L^2}^2 + \delta\|u\|^{\alpha_2+2}_{L^{\alpha_2+2}}),
\end{equation}
and thus, we choose $\delta=\delta_{4C}>0$ such that the following condition holds 
\begin{equation}\label{E:delta-4C}
\delta_{4C} \leq \frac{|\lambda_2|}{|\lambda_1|}\frac{(\alpha_3 - \alpha_2)}{(\alpha_3 - \alpha_1)} \frac{(\alpha_1 +2)}{(\alpha_2 + 2)}. 
\end{equation}
Thus,
\begin{align*}
    y^\prime(t) &\geq \left[\frac{\alpha_3}{2} - 2\right]\|u_x\|^2_{L^2} + \frac{\lambda_1(\alpha_3-\alpha_1)}{\alpha_1 + 2}(C(\delta)\|u\|_{L^2}^2 + \delta\|u\|^{\alpha_2+2}_{L^{\alpha_2+2}}) + \frac{\lambda_2(\alpha_3 - \alpha_2)}{\alpha_2 + 2}\|u\|^{\alpha_2 + 2}_{L^{\alpha_2 + 2}} - E\alpha_3 \\
    &\geq \left[\frac{\alpha_3}{2} - 2\right]\|u_x\|^2_{L^2} + \frac{\lambda_1C(\delta)(\alpha_3-\alpha_1)}{\alpha_1 + 2}M - E\alpha_3>0,
\end{align*}
provided $E < - \frac{|\lambda_1|C(\delta_{4C})(\alpha_3-\alpha_1)}{\alpha_3(\alpha_1 + 2)}M$ and $\alpha_3 >4$ by the assumption in this case.

Therefore, we have shown that $y^\prime(t) > 0$ in all cases, which concludes the proof of Theorem \ref{NegE_Bup_3term}. \qed

\begin{remark}\label{RMKGeneralize}
    In the same spirit, \textbf{Case 1} and \textbf{Case 2A} easily generalize for finitely many combined nonlinearities. Other cases can be more challenging to adapt, especially when there are various coefficients with alternating signs.
\end{remark}

\section{Blow-up for an infinite sum}\label{S:Proof_exp}
Upon computing the cases for Theorem \ref{NegE_Bup_3term}, we noticed that techniques from \textbf{Case 2A} and \textbf{Case 2C} could be combined to develop a criterion for an exponential nonlinearity, namely,
\begin{equation}{\label{ExpNLS}}
    \begin{cases}
        &iu_{t} + \partial_{x}^{2}u = - \sum_{k=0}^{\infty}d_k|u|^{\alpha_k}u, \\
        &u(x,0) = u_{0}, \\
    \end{cases}
\end{equation}
 where $(x,t) \in \R\times \R^+$, $d_k = \frac{1}{k!}$ and $\alpha_k = k$.

Using the notation from Section \ref{S:Prelim}, we re-introduce the energy and variance adjusted for the exponential nonlinearity
\begin{equation}
    E[u] = \frac{1}{2}\|u_x\|_{L^2}^2 - \sum_{k=0}^\infty \frac{1}{(k+2)k!} \|u\|_{L^{k+2}}^{k+2},
\end{equation}
and 
\begin{equation}
    V_{tt}[u] = -4y^\prime \equiv 8 \|u_x\|_{L^2}^2 - 4\sum_{k=1}^\infty \frac{k}{(k+2)k!} \|u\|_{L^{k+2}}^{k+2}.
\end{equation}

\subsection{Proof of Theorem \ref{NegE_Bup_Exp}}
The main idea for this argument is to break this sum into parts. The main difficulty with this technique is controlling the subcritical nonlinearities. Grouping and estimating them as in \textbf{Case 2C} of Theorem \ref{NegE_Bup_3term}, the remaining supercritical nonlinearities can be easily worked around as in \textbf{Case 2A}. Splitting the sum into subcritical, critical and supercritical powers, we write
\begin{align*}
    y^\prime &= -2\|u_x\|_{L^2}^2 + \sum_{k=1}^4 \frac{k}{(k+2)k!} \|u\|_{L^{k+2}}^{k+2} + \frac{5}{7(5!)} \|u\|_{L^{7}}^{7} + \frac{6}{8(6!)} \|u\|_{L^{8}}^{8}+ \sum_{k=7}^\infty \frac{k}{(k+2)k!} \|u\|_{L^{k+2}}^{k+2}.
   \end{align*}
Recalling the energy, we separate the $L^7$ norm, and then substitute into the above expression
\begin{equation}\label{L^6_energy}
     \frac{1}{7(5!)}\|u\|^{7}_{L^{7}} =\frac{1}{2} \|u_x\|^2_{L^2} - \frac{1}{2}M - \sum_{k=1}^4 \frac{1}{(k+2)k!} \|u\|_{L^{k+2}}^{k+2} - \frac{1}{8(6!)} \|u\|_{L^{8}}^{8} - \sum_{k=7}^\infty \frac{1}{(k+2)k!} \|u\|_{L^{k+2}}^{k+2} - E,
\end{equation}
\begin{equation}\label{E:y1}
    y^\prime = \frac{5-4}{2}\|u_x\|_{L^2}^2 + \sum_{k=1}^4 \underbrace{\overbrace{\frac{ (k-5)}{(k+2)k!}}^{A_k} \|u\|_{L^{k+2}}^{k+2}}_{\star}  + \frac{6-5}{8(6!)} \|u\|_{L^{8}}^{8} - \frac{5}{2}M +\sum_{k=7}^\infty \frac{ (k-5)}{(k+2)k!} \|u\|_{L^{k+2}}^{k+2} - 5 E,
\end{equation}
where we used the {$k=6$} term of the sum to cancel with $\delta_k\|u\|_8^8$.  
We now use the interpolation \eqref{E:Youngs_general} to upper bound the terms in the first sum
\begin{equation}\label{E:1}
    |\star| \leq |A_k|(C(\delta_k)\|u\|_2^2 + \delta_k\|u\|_8^8).
\end{equation}
For each $k = 1,2,3,4$, choose $\delta_k$ such that  
\begin{equation}\label{E:2a}
\delta_k|A_k| \leq \frac{1}{4} \frac1{8(6!)}.
\end{equation}

Since $A_1 = -\frac{4}{3}, \; A_2 = -\frac{3}{8}, \; A_3 = -\frac{1}{15}, \; A_4 = -\frac{1}{144}$, 
taking $\delta_1 \leq \frac{1}{30720}$, $\delta_2 \leq \frac{1}{8640}$, $\delta_3 \leq \frac{1}{1536}$, $\delta_4 \leq \frac{1}{160}$, would satisfy the inequality \eqref{E:2a}. 
Then for each $\delta_k$ we can take the sharp constant $C_k$ such that $a^{k+2} \leq C_k a^2 + \delta_k a^8$ for any $a > 0$ (see Remark \ref{R:2.2}); thus, we get $C_1 \geq \frac{5(160)^{1/5}}{3},  C_2 \geq 16\sqrt{5}, C_3 \geq  384, \text{ and } C_4 \geq  {\frac{102400}{27}}$.

Substituting these bounds into \eqref{E:y1}, we obtain
\begin{align*}
    y^\prime &\geq \frac{1}{2}\|u_x\|_{L^2}^2 - \sum_{k=1}^4 \frac{1}{4}\frac{1}{8(6!)}\|u\|_{L^8}^8 + \frac{1}{8(6!)} \|u\|_{L^{8}}^{8} 
 +\sum_{k=7}^\infty \frac{(k-5)}{(k+2)k!} \|u\|_{L^{k+2}}^{k+2}\\ 
 &\hspace{4cm} + \sum_{k=1}^4\frac{ (k-5)}{(k+2)k!}C(\delta_k)M -  \frac{5}{2}M - 5 E \\
    &\geq \frac{1}{2}\|u_x\|_{L^2}^2 -\Bigg[ \underbrace{\left(\sum_{k=1}^4\frac{ (5-k)}{(k+2)k!}C(\delta_k) + \frac{5}{2}\right)}_{\kappa^\prime}M + 5 E\Bigg],
\end{align*}

Therefore, if the energy is chosen such that 
\begin{equation}\label{E:EM-kappa1}
E[u_0] < \frac15 \kappa^\prime M[u_0] \equiv -\kappa \, M[u_0],
\end{equation}
then we obtain $y' >0$. Here, 
\begin{equation}\label{E:kappa}
\kappa = \frac{1}{5}\kappa^\prime \approx \frac{(73.986)}{5} \approx 14.797.
\end{equation}
Rewriting \eqref{E:EM-kappa1} with all terms on the left side, we obtain the condition \eqref{E:condition-exp}, thus, finishing the proof of Theorem \ref{NegE_Bup_Exp}.  \qed

\subsection{A variation of the technique.}
Here we show how to use this method not relying on the extra supercritical power. 
In that case, we separate the sum into two parts, with the second one starting from the next smallest supercritical power present in the sum. Using the example of an exponential, we separate into subcritical/critical term, a supercritical term (split into 2 parts), and the rest of the sum:
\begin{align*}
    y^\prime &= -2 \|u_x\|_{L^2}^2 + \sum_{k=1}^\infty \frac{k}{(k+2)k!} \|u\|_{L^{k+2}}^{k+2} \\
    &\geq -2\|u_x\|_{L^2}^2 + \sum_{k=1}^4 \frac{\theta k}{(k+2)k!} \|u\|_{L^{k+2}}^{k+2} + \frac{5\theta}{7(5!)} \|u\|_{L^{7}}^{7} + \frac{5(1-\theta)}{7(5!)} \|u\|_{L^{7}}^{7}+ \sum_{k=6}^\infty \frac{\theta k}{(k+2)k!} \|u\|_{L^{k+2}}^{k+2}.
\end{align*}

Similarly to the proof of Theorem \ref{NegE_Bup_3term}, we use the substitution
\begin{equation}\label{L^7_energy}
     \frac{1}{7(5!)}\|u\|^{7}_{L^{7}} = \left[\frac{1}{2} \|u_x\|^2_{L^2} - \frac{1}{2}M - \sum_{k=1}^4 \frac{1}{(k+2)k!} \|u\|_{L^{k+2}}^{k+2} - \sum_{k=6}^\infty \frac{1}{(k+2)k!} \|u\|_{L^{k+2}}^{k+2} - E\right],
\end{equation}
to obtain

\begin{equation*}
    y^\prime \geq \underbrace{\left[\frac{5\theta}{2}-2\right]}_{\epsilon}\|u_x\|_{L^2}^2 + \theta\sum_{k=1}^4 \underbrace{\overbrace{\frac{ (k-5)}{(k+2)k!}}^{A_k} \|u\|_{L^{k+2}}^{k+2}}_{\star}  + \frac{(1-\theta)}{7(4!)} \|u\|_{L^{7}}^{7} - \frac{5\theta}{2}M +\sum_{k=6}^\infty \frac{\theta (k-5)}{(k+2)k!} \|u\|_{L^{k+2}}^{k+2} - 5\theta E.
\end{equation*}
We require $\frac{5\theta - 4}{2} > 0$, thus, $\theta > \frac{4}{5}$, i.e., $\frac{4}{5} < \theta < 1$. Then notice that, for $k = 1,2,3,4$, we can bound
\begin{equation}\label{E:1b}
    |\star| \leq |A_k|(C(\delta_k)\|u\|_2^2 + \delta_k\|u\|_7^7).
\end{equation}
Since there are 4 terms in the sum with $A_k$'s, we choose each $\delta_k$ such that for $k = 1,2,3,4$, 
\begin{equation}\label{E:2b}
\delta_k|A_k| \leq \frac{1}{4} \frac1{7(4!)}  \frac{(1-\theta)}{\theta}. 
\end{equation} 
Combining the estimates, we obtain

\begin{align*}
    y^\prime &\geq \epsilon\|u_x\|_{L^2}^2 - \sum_{k=1}^4 \frac{1}{4}\frac{(1-\theta)}{7(4!)}\|u\|_{L^7}^7 + \frac{(1-\theta)}{7(4!)} \|u\|_{L^{7}}^{7} 
 +\sum_{k=6}^\infty \frac{\theta (k-5)}{(k+2)k!} \|u\|_{L^{k+2}}^{k+2}\\ 
 &\hspace{4cm} + \sum_{k=1}^4\frac{\theta (k-5)}{(k+2)k!}C(\delta_k)M -  \frac{5\theta}{2}M - 5\theta E \\
    &\geq \epsilon\|u_x\|_{L^2}^2 -\theta\left( \sum_{k=1}^4\frac{ (5-k)}{(k+2)k!}C(\delta_k)M + \frac{5}{2}M + 5 E\right).
\end{align*}

For simplicity of the statement, choose $\kappa_1 = \max\{C(\delta_i), i=1,2,3,4\}$, then 
$$
\sum_{k=1}^4\frac{ (k-5)}{(k+2)k!}C(\delta_k) \leq - \frac{1283}{720} \kappa_1 \approx -1.78 \kappa_1.
$$ 
We conclude that $y^\prime(t) > 0$ for all time $t \in (0,T^*)$, provided the following condition holds
$$ 
E[u_0] < -\left(\frac{1283}{720} \kappa_1 + 2 \right)M[u_0] \quad \mbox{or} \quad E[u_0]+ \left(\frac{1283}{720} \kappa_1 + 2 \right) M[u_0] < 0,
$$
which gives an alternative proof to 
Theorem \ref{NegE_Bup_Exp} with a slightly larger constant, but it shows how to handle the splitting not relying on an extra supercritical term. \qed

\subsection{A proof of Corollary \ref{Generalization}}
With some restriction on coefficients $d_k$, we notice that Theorem \ref{NegE_Bup_Exp} can easily be generalized for any converging infinite sum or a finite number of terms, provided certain conditions on the powers and the coefficients are posed. Indeed, let the nonlinearity $\mathcal{N}(u)  = \sum_{k=1}^N d_k|u|^{\alpha_k}$, $0<N\leq \infty$, be such that $d_k \geq 0 \; \forall k \in \{1,2,...,N \}$ and for some $k^*$, $d_{k^*} > 0$, $\alpha_{k^*} > 4$ and $\alpha_1 < \alpha_2 < ... < \alpha_{k^*} < ... < \alpha_N$. For a finite number of nonlinear terms, it is more optimal to take $\alpha_{k^*}$ to be the first occurrence of a \textit{super-critical} exponent.
 If $\alpha_{k^*} > 4$ and is the last term in the sum, a special case for this proof is outlined in the subsection \ref{Special}. 
We proceed with showing that $y^\prime > 0$ for the Cauchy problem in \eqref{ExpNLS} in the usual way.
\begin{align*}
    y^\prime  &= -2 \|u_x\|_{L^2}^2 + \sum_{k=1}^N d_k \frac{\alpha_k}{\alpha_k + 2 }\|u\|_{L^{\alpha_k+2}}^{\alpha_k + 2} \\ 
    &= -2 \|u_x\|_{L^2}^2 + \sum_{k=1}^{k^*-1} d_k \frac{\alpha_k}{\alpha_k + 2 }\|u\|_{L^{\alpha_k+2}}^{\alpha_k + 2} + d_{k^*} \frac{\alpha_{k^*}}{\alpha_{k^*} + 2 }\|u\|_{L^{\alpha_{k^*}+2}}^{\alpha_{k^*} + 2} + \sum_{k=k^*+1}^N d_k \frac{\alpha_k}{\alpha_k + 2 }\|u\|_{L^{\alpha_k+2}}^{\alpha_k + 2}.
\end{align*}
Recalling the energy, we use the friendly substitution for the norm with $k^\ast$
\begin{equation}\label{kstarsub}
    \frac{d_{k^*} }{\alpha_{k^*} + 2 }\|u\|_{L^{\alpha_{k^*}+2}}^{\alpha_{k^*} + 2} = \frac{1}{2}\|u_x\|_{L^2}^2 - \sum_{k=1}^{k^*-1}  \frac{d_k}{\alpha_k + 2 }\|u\|_{L^{\alpha_k+2}}^{\alpha_k + 2} - \sum_{k=k^*+1}^N \frac{d_k}{\alpha_k + 2 }\|u\|_{L^{\alpha_k+2}}^{\alpha_k + 2} - E
\end{equation}
to write
\begin{equation}\label{general_step1}
    y^\prime = \left[\frac{\alpha_{k^*}-4}{2}\right] \|u_x\|_{L^2}^2 + \sum_{k=1}^{k^*-1} d_k \frac{\alpha_k - \alpha_{k^*}}{\alpha_k + 2 }\|u\|_{L^{\alpha_k+2}}^{\alpha_k + 2} + \sum_{k=k^*+1}^N d_k \frac{\alpha_k- \alpha_{k^*}}{\alpha_k + 2 }\|u\|_{L^{\alpha_k+2}}^{\alpha_k + 2} - E\alpha_{k^*}.
\end{equation}

Note that any finite number of coefficients $d_k$ can be zero (except for $d_{k^\ast}$) and will not change the computation. Also, starting from $k^\ast +2$ index, the coefficients $d_k$ (the tail) can be zero. 
\smallskip

To proceed, we notice that the third term in \eqref{general_step1} is strictly positive, providing the opportunity to bound it below by $0$. However, the second term is strictly negative and we would like to use some of the positive terms from the third part to help control the negative ones. For the finite case, we require at least $N \geq k^*+1$. In the infinite case, we choose to keep the first term of the sum and bound away the rest, but note that the choice is arbitrary and we may choose to keep any finite number of terms. Thus, we have
\begin{align*}
    y^\prime &\geq \left[\frac{\alpha_{k^*}-4}{2}\right] \|u_x\|_{L^2}^2 + \underbrace{\sum_{k=1}^{k^*-1} \overbrace{d_k \frac{\alpha_k - \alpha_{k^*}}{\alpha_k + 2 }}^{A_k}\|u\|_{L^{\alpha_k+2}}^{\alpha_k + 2}}_{(\star)} + d_{k^* + 1} \frac{\alpha_{k^* + 1}- \alpha_{k^*}}{\alpha_{k^* + 1} + 2 }\|u\|_{L^{\alpha_{k^* + 1}+2}}^{\alpha_{k^* + 1} + 2} - E\alpha_{k^*}
\end{align*}
For $k = 1,2,...,k^*-1$, we bound the $(\ast)$ term  via interpolation
\begin{equation}\label{E:1Inf}
    |(\star)| \leq |A_k|(C(\delta_k)\|u\|_2^2 + \delta_k\|u\|_{L^{\alpha_{k^* + 1}+2}}^{\alpha_{k^* + 1} + 2}).
\end{equation}
Choose each $\delta_k$ such that for $k = 1,2,...,k^*-1$, we get
\begin{equation}\label{E:2inf}
\delta_k|A_k| \leq \frac{1}{k^*-1}  \frac{d_{k^* + 1}(\alpha_{k^* + 1}- \alpha_{k^*})}{\alpha_{k^* + 1} + 2 }. 
\end{equation} 
Using the fact that $\alpha_k - \alpha_{k^*} < 0$ for $k < k^*$, we deduce
\begin{align*}
    y^\prime &\geq \left[\frac{\alpha_{k^*}-4}{2}\right] \|u_x\|_{L^2}^2 - \sum_{k=1}^{k^*-1} \frac{1}{k^*-1}  \frac{d_{k^* + 1}(\alpha_{k^* + 1}- \alpha_{k^*})}{\alpha_{k^* + 1} + 2 } + d_{k^* + 1} \frac{\alpha_{k^* + 1}- \alpha_{k^*}}{\alpha_{k^* + 1} + 2 }\|u\|_{L^{\alpha_{k^* + 1}+2}}^{\alpha_{k^* + 1} + 2}\\ &\hspace{4.5cm} + \sum_{k=1}^{k^*-1}d_k \frac{\alpha_k - \alpha_{k^*}}{\alpha_k + 2 }C(\delta_k)M- E\alpha_{k^*} \\
    &= \left[\frac{\alpha_{k^*}-4}{2}\right] \|u_x\|_{L^2}^2 + \sum_{k=1}^{k^*-1}d_k \frac{\alpha_k - \alpha_{k^*}}{\alpha_k + 2 }C(\delta_k)M- E\alpha_{k^*}.
\end{align*}
Therefore, $y^\prime > 0$, provided
\begin{equation*}
     E\alpha_{k^*} + \sum_{k=1}^{k^*-1}d_k \frac{ \alpha_{k^*}-\alpha_k }{\alpha_k + 2 }C(\delta_k)M < 0, 
\end{equation*}
concluding our proof. \qed

\subsubsection{The special case.}\label{Special}
The computation proceeds differently in the case where $\alpha_{k^*} > 4$ and is the last term in the sum. Here, we have
\begin{align*}
    y^\prime  &= -2 \|u_x\|_{L^2}^2 + \sum_{k=1}^{k^*} d_k \frac{\alpha_k}{\alpha_k + 2 }\|u\|_{L^{\alpha_k+2}}^{\alpha_k + 2} \\ 
    &= -2 \|u_x\|_{L^2}^2 + \sum_{k=1}^{k^*-1} d_k \frac{\alpha_k}{\alpha_k + 2 }\|u\|_{L^{\alpha_k+2}}^{\alpha_k + 2} + d_{k^*} \frac{\alpha_{k^*}\theta}{\alpha_{k^*} + 2 }\|u\|_{L^{\alpha_{k^*}+2}}^{\alpha_{k^*} + 2} + d_{k^*} \frac{\alpha_{k^*}(1-\theta)}{\alpha_{k^*} + 2 }\|u\|_{L^{\alpha_{k^*}+2}}^{\alpha_{k^*} + 2} .
    \end{align*}
Then using \eqref{kstarsub}, for $\frac{4}{\alpha_{k^*}} < \theta < 1$,
\begin{equation}\label{general_step1b}
    y^\prime = \left[\frac{\alpha_{k^*}\theta-4}{2}\right] \|u_x\|_{L^2}^2 + \sum_{k=1}^{k^*-1} \underbrace{\overbrace{d_k \frac{(\alpha_k - \alpha_{k^*}\theta)}{\alpha_k + 2 }}^{A_k}\|u\|_{L^{\alpha_k+2}}^{\alpha_k + 2}}_{(\star)} + d_{k^*} \frac{\alpha_{k^*}(1-\theta)}{\alpha_{k^*} + 2 }\|u\|_{L^{\alpha_{k^*}+2}}^{\alpha_{k^*} + 2} - E\alpha_{k^*}\theta.
    \end{equation}
Note that, for $k = 1,2,...,k^*-1$, we interpolate to bound the $(\ast)$
\begin{equation}\label{E:1Infb}
    |(\star)| \leq |A_k|(C(\delta_k)\|u\|_2^2 + \delta_k\|u\|_{L^{\alpha_{k^*}+2}}^{\alpha_{k^*} + 2}).
\end{equation}
Choose each $\delta_k$ such that for $k = 1,2,...,k^*-1$, we have
\begin{equation}\label{E:2infb}
\delta_k|A_k| \leq \frac{d_{k^*}}{k^*-1}  \frac{\alpha_{k^*}(1-\theta)}{\alpha_{k^*} + 2 }. 
\end{equation}
Using the fact that $\alpha_k - \alpha_{k^*}\theta < 0$ for $k < k^*$, we deduce
\begin{align*}
    y^\prime &\geq  \left[\frac{\alpha_{k^*}-4}{2}\right] \|u_x\|_{L^2}^2 + \sum_{k=1}^{k^*-1}d_k \frac{\alpha_k - \alpha_{k^*}\theta}{\alpha_k + 2 }C(\delta_k)M- E\alpha_{k^*}\theta,
\end{align*}
and hence, $y^\prime > 0$, provided
\begin{equation*}
     E + \frac{1}{\alpha_{k^*}\theta}\sum_{k=1}^{k^*-1}d_k \frac{ \alpha_{k^*}\theta -\alpha_k }{\alpha_k + 2 }C(\delta_k)M < 0, 
\end{equation*}
concluding the proof of Corollary \ref{Generalization}. \qed

\section{Blow-up with positive energy}\label{S:Proof_posEnergy}
In this section, we develop a blow-up criterion in the style of {Holmer-Platte-Roudenko} \cite{HRP2010}, or its generalization by Duyckaerts-Roudenko \cite{DR2015}. We consider a double nonlinearity
\begin{equation}{\label{CombinedNLS}}
    \begin{cases}
        &iu_{t} + \partial_{x}^{2}u + \epsilon_1|u|^{p-1}u + \epsilon_2|u|^{q-1}u = 0, \\
        &u(x,0) = u_{0} \in H^1(\mathbb R), \\
    \end{cases}
\end{equation}
 where $(x,t) \in \R\times \R$, $\epsilon_1 < 0$ (the smaller nonlinearity is defocusing), $\epsilon_2 > 0$ (the larger nonlinearity is focusing), $q>5$ and $1 < p < q$. We begin with some useful inequalities, then proceed with the simpler case of $p=5$. For brevity, we write $M = M[u_0]$ and $E = E[u_0]$ (as these quantities are conserved for a given $u_0$), we also will occasionally use abbreviation $V$ for the variance $V[u](t)$.

\subsection{Useful Inequalities}

The first one is the Gagliardo-Nirenberg inequality in 1D for $p > 1$, 
\begin{equation}\label{GNIneq}
    \norm{u}^{p+1}_{L^{p+1}} \leq C_{GN}\norm{u_x}^{\frac{p-1}{2}}_{L^2}\norm{u}^{\frac{p-1}{2}+2}_{L^2}, 
\end{equation}
where $C_{GN}$ is a sharp constant (e.g., see \cite{W1983}, \cite{DR2015}).

The next inequality is obtained from \cite[Proposition 4.3]{DR2015} where it is presented and directly proved for dimensions $N \geq 1$. In \cite{HRP2010}, for the case $p = 3$. For our purposes, it is sufficient to consider the case for one dimension,
\begin{equation}\label{DRIneq}
    \norm{u}_{L^2} \leq \tilde C\left(\norm{xu}^{\frac{p-1}{2}}_{L^2}\norm{u}^{p+1}_{L^{p+1}}\right)^{\frac{2}{3p+1}},
\end{equation}
where $\tilde C$ denotes
{\small
\begin{equation}\label{DRIneqConst}
    \tilde C = \left(\frac{3p+1}{2(p+1)}\right)^{\frac{2}{3p+1}}
    \left(\frac{\pi(3p+1)}{p-1}\right)^{\frac{p-1}{2(3p+1)}}\left(\frac{\Gamma(\frac{p+1}{p-1})}{\Gamma(\frac{p+1}{p-1} + \frac{1}{2})}\right)^{\frac{p-1}{3p+1}}.
\end{equation}
}
Note that from \eqref{DRIneq}, the definitions of mass \eqref{E:Mass} and variance \eqref{Variance} and denoting $C^* = \tilde C^{\frac{3p+1}{2}}$, we obtain
\begin{equation}\label{DRIneqChange}
    \frac{M^{\frac{p-1}{4} + \frac{p+1}{2}}}{V(t)^\frac{p-1}{4}C^*} \leq \|u(t)\|^{p+1}_{L^{p+1}}. 
\end{equation}
We also make use of the uncertainty principle for the solution $u(t)$
\begin{equation}\label{GradMassVar}
    \norm{u_x(t)}^2_{L^2} \geq \frac{1}{4}\left[\frac{M^2}{V(t)} + \frac{V_t^2(t)}{4V(t)}\right],
\end{equation}
the derivation of which is given, for example, in \cite[Section 4.1]{DR2015}. 

\subsection{\textbf{Proof of Theorem \ref{DR_Bup}}}

 \subsubsection{\bf \underline{Critical $p=5$}}
 Recall the virial identity, which we deduce from \eqref{Vdt2} with $\alpha_1 = p-1$, $\alpha_2 = q-1$, $d_1 = \epsilon_1$ and $d_2 = \epsilon_2$ (the rest of $\alpha_j, d_j = 0$), and for conciseness we drop the time dependence to obtain
 \begin{equation}
    V_{tt} = 8\lVert u_x \rVert^2_{L^2} - \frac{4\epsilon_1(p-1)}{p+1}\lVert  u \rVert^{p+1}_{L^{p+1}} - \frac{4\epsilon_2(q-1)}{q+1}\lVert u \rVert^{q+1}_{L^{q+1}}.
\end{equation}

We first study the case with $p=5, \; q>5, \; \epsilon_1 < 0,$ and $\epsilon_2 > 0$. We re-write the energy and second derivative of the variance as follows
\begin{equation}\label{p5Energy}
E[u] = \frac{1}{2}\| u_x\|_{L^2}^2 + \frac{|\epsilon_1|}{6}\|u\|_{L^{6}}^{6} - \frac{\epsilon_2}{q+1}\|u\|_{L^{q+1}}^{q+1}
\end{equation}
and 
\begin{equation}\label{p5Variancett}
V_{tt}[u] = 8\| u_x\|_{L^2}^2 + \frac{8|\epsilon_1|}{3}\|u\|_{L^{6}}^{6} - \frac{4\epsilon_2(q-1)}{q+1}\|u\|_{L^{q+1}}^{q+1}.
\end{equation}
(Notice that if we would have $\epsilon_2 < 0$ (besides $\epsilon_1 <0$) and $q > 1$, then $V_{tt}$ would never be negative.) 

Re-scaling the energy to a factor of $4(q-1)$ and solving for the $\|u\|_{L^{q+1}}$ term, we obtain
\begin{equation*}
     \frac{4(q-1)\epsilon_2}{q+1}\|u\|_{L^{q+1}}^{q+1}= 2(q-1)\| u_x\|_{L^2}^2 + \frac{4(q-1)|\epsilon_1|}{6}\|u\|_{L^{6}}^{6}-4(q-1)E[u].
\end{equation*}
Substituting this into \eqref{p5Variancett} and simplifying, we get
\begin{equation*}
V_{tt}[u] = 2\left[5- q\right]\| u_x\|_{L^2}^2 + \frac{2|\epsilon_1|}{3}\left[5-q\right]\|u\|_{L^{6}}^{6} +4(q-1)E[u].
\end{equation*}
Notice that if $E<0$, then $V_{tt}<0$ immediately. Thus, we assume for the following estimates that $E > 0$. Using inequalities \eqref{DRIneqChange}
and \eqref{GradMassVar} with $C^* = \frac43\pi^2$ as defined in \eqref{DRIneqConst}, we obtain
\begin{align*}
    V_{tt} &\leq 4(q-1)E-\frac{(q-5)}{2}\left[\frac{M^2}{V} + \frac{V_t^2}{4V}\right] - \frac{2|\epsilon_1|(q-5)}{3}\left[\frac{M^4}{VC^*}\right] \\
    &= 4(q-1)E - \frac{(q-5)M^2}{2V}\left[1 + \frac{4|\epsilon_1|M^2}{3C^*}\right] -\frac{(q-5)}{8}\frac{V_t^2}{V}.
\end{align*}
Denoting the expression in the square brackets by $H$, i.e., $H:=1 + \frac{4|\epsilon_1|M^2}{3C^*}$, we re-write
\begin{align}\label{E:Vtt}
    VV_{tt} + \frac{(q-5)}{8}V_t^2 &\leq  4(q-1)EV - \frac{(q-5)}{2} H M^2.
\end{align}
Let $\alpha = \frac{q-5}{8}$ and 
observe that if $V^{\alpha+1} =B$, then $B_{tt} = (V^{\alpha+1})_{tt} = (\alpha+1) V^{\alpha-1} [VV_{tt} + \alpha V_t^2]$, which is the left-hand side in \eqref{E:Vtt} multiplied by $(\alpha+1) V^{\alpha-1}$.
Thus, we can re-write \eqref{E:Vtt} with the 
rescaling
\begin{itemize}
    \item []
    \begin{equation}\label{B}
        V = B^{\frac{1}{\alpha + 1}},
    \end{equation}
\item[]
\begin{equation}\label{E:zeros}
B_{tt} \leq (\alpha+1) B^{\frac{\alpha-1}{\alpha+1}} 
\Big(4(q-1)E B^{\frac1{\alpha+1}} - 4\alpha H M^2 \Big),
\end{equation}

\end{itemize}
and with some simplification, it yields
\begin{align}\label{BttUpperBound}
    B_{tt}  &\leq  4(\alpha + 1)(q-1)EB^{\frac{\alpha}{\alpha + 1}} - 4 \alpha(\alpha + 1) M^2HB^{\frac{\alpha -1}{\alpha + 1}}.
\end{align}

Notice that if $B_{tt}$ is bounded above by the RHS of \eqref{BttUpperBound}, then we may introduce a non-negative perturbation function $-g^2(t)$ (which is squared to {ensure the non-negativity}) that yields the equality
\begin{align}\label{BttEquality}
    B_{tt}  &=  4(\alpha + 1)(q-1)EB^{\frac{\alpha}{\alpha + 1}} - 4\alpha(\alpha + 1) M^2HB^{\frac{\alpha -1}{\alpha + 1}} - g^2(t).
\end{align}

There is a relevant physical interpretation that involves Newtons second law $F_{net} = ma$. Indeed take the mass of a particle $m = 1$, and recall that acceleration is simply the second derivative of the position of that particle $x_{tt}$. Then we have
\begin{equation}\label{Newton2ndLaw}
    x_{tt} = F_1 + ... + F_n, 
\end{equation}
where we write out $F_{net}$ as the sum of forces for emphasis. Since (in our case) $B_{tt}$ is a re-scaling of the variance $V$, we view the variance as the position of the particle, and hence, we have two forces seemingly contributing to its \textit{acceleration}, namely, 
\begin{equation}\label{Force1}
    F_1 = 4(\alpha + 1)(q-1)EB^{\frac{\alpha}{\alpha + 1}} - 4\alpha(\alpha + 1) M^2HB^{\frac{\alpha -1}{\alpha + 1}},
\end{equation}
and
\begin{equation}\label{Force2}
    F_2 = -g^2(t).
\end{equation}

Recall again that $B$ is a rescaling of the variance $V$, whereas if $V \to 0$, then $B \to 0$. We consider that $B \to 0$ if $B_{tt} < 0$, which depends on the signs of \eqref{Force1} and \eqref{Force2}. Notice that if $F_1 < 0$, then by our construction of $F_2$ (i.e., $g^2(t)$ is {nonnegative}, thus, $F_2$ always pulls the particle to the origin), $B \to 0$ surely. However, if $F_1 > 0$, then we must introduce a criterion to ensure $B \to 0$. To do so, we write $F_1$ in terms of the potential energy 
$$F_1 = -\frac{\partial U}{\partial B},$$
which acts like an attracting force towards the origin. We then write an analogous version of \eqref{Newton2ndLaw},
\begin{equation}\label{BversionNewtons2ndLaw}
    B_{tt} + \frac{\partial U}{\partial B} = -g^2(t).
\end{equation}
Multiplying by $B_t$ and integrating in time, we get
\begin{equation}\label{ParticleEnergyBehavior}
   \frac{ B_{t}^2}{2} + U(B) = -\int g^2(t)B_t dt.
\end{equation}
We then write that the energy $\mathcal{E}$ of the particle is given by
\begin{equation}\label{ParticleEnergy}
    \mathcal{E}(t) = \frac{ B_{t}^2}{2} + U(B). 
\end{equation}
Recall, the second force \eqref{Force2} only aids in pulling our particle towards the origin, thus it is sufficient to analyze only $B_{tt} = F_1$, for which \eqref{ParticleEnergy} is conserved by \eqref{ParticleEnergyBehavior}.
Integrating over $B$ in $F_1$, we obtain
\begin{equation}\label{Potential}
    U(B) = -16(\alpha+1)^2 E B^{\frac{2\alpha + 1}{\alpha + 1}} + 2(\alpha+1)^2M^2HB^{\frac{2\alpha}{\alpha + 1}},  
\end{equation}
where any extra constant gained is zero, since we consider $U(0) = 0$. 

We then write the energy of the particle $B$ in \eqref{ParticleEnergy} as
\begin{equation}\label{ParticleEnergy_B_p5}
    \mathcal{E}(t) = \frac{ B_{t}^2}{2} 
    - 16 (\alpha+1)^2
    E B^{\frac{2\alpha + 1}{\alpha + 1}} + 2(\alpha+1)^2M^2HB^{\frac{2\alpha}{\alpha + 1}}. 
\end{equation}

We proceed by outlining the motivation for our blow up criteria as suggested in \cite{Lushnikov1995} and discussed with further details in \cite{HRP2010}, \cite{DR2015}.

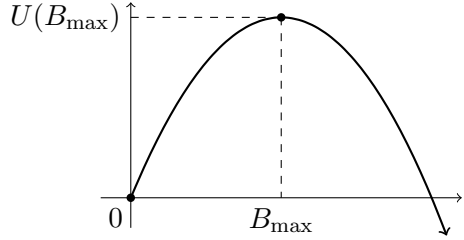
\begin{figure}[H]
    \centering
    \begin{tikzpicture}[scale=4]
        
        \draw[->] (-0.1,0) -- (1.1,0) node[right] {};
        \draw[->] (0,-0.1) -- (0,0.65) node[above] {};

        \draw[thick, domain=0:1.05, smooth, variable=\x, ->] 
            plot ({\x}, {0.6*(1 - (2*\x - 1)^2)});

        \fill (0,0) circle (0.4pt);
        
        \fill (0.5,0.6) circle (0.4pt);

        \draw[dashed] (0.5,0.6) -- (0.5,0);
        \draw[dashed] (0,0.6) -- (0.5,0.6);

        \node[below] at (-.05,0) {$0$};
        \node[below] at (1,0) {};
        \node[below] at (0.5,0) {$B_{\text{max}}$};
        \node[left] at (0,0.6) {$U(B_{\text{max}})$};

    \end{tikzpicture}
    \caption{Sketch of $U(B)$ with maximum at $B_{\text{max}}$}
    \label{F:PotentialEnergyGraph}
\end{figure}
Figure \ref{F:PotentialEnergyGraph} is a sketch of the potential \eqref{Potential}. If the particle starts to the left of the bump, $B < B_{max}$, it would never overcome the maximum point, if it does not have enough energy: $\mathcal{E}(0) < U(B_{max})$, and thus, $B \to 0$. However, if the particle starts from the right of the bump, $B > B_{max}$, we have two possibilities:
\begin{enumerate}
    \item If we have enough energy $\mathcal{E}(0) > U(B_{max})$ and $B_t < 0$ (that is, the particle is moving to the left), then the particle can overcome the bump, and $B \to 0$. 
    \item If the particle does not have enough energy $\mathcal{E}(0) {\leq} U(B_{max})$, or an incorrect direction of motion $B_t  { > } 0$, we cannot conclude that collapse occurs.
\end{enumerate}
We can organize the above conditions to obtain {three} sufficient criteria for collapse:
\begin{center}
\begin{varwidth}{\textwidth}
    \begin{enumerate}
  \item[{\bf(I)}]$\mathcal{E}(0) < U(B_{max})$ and $B(0) < B_{max}$, \hspace{4cm}
    \item[{\bf (II)}] $\mathcal{E}(0) > U(B_{max})$ and $B_t(0) < 0,$ \hspace{4cm}
    {\item[{\bf (III)}] $\mathcal{E}(0) = U(B_{max})$, $B(0) < B_{max}$, and $B_t(0) < 0.$}
    \end{enumerate}
\end{varwidth}
\end{center}

To write out conditions explicitly, we must first compute $B_{max}$, which can be done by finding the critical points of \eqref{Potential} in the typical fashion (e.g., finding zeros of \eqref{Force1} or \eqref{E:zeros}), for which we obtain
\begin{align*}
 4(\alpha + 1) B^{\frac{\alpha-1}{\alpha + 1}}\left( (q-1)EB^{\frac{1}{\alpha + 1}} - \alpha M^2H\right) = 0,
\end{align*}
and thus, the maximum value is attained at
\begin{align}\label{E:B_max_p5}
    B_{max} = \left(\frac{\alpha H M^2}{(q-1)E}\right)^{\alpha + 1}.
\end{align}
Hence, the maximum value for our potential $U$ is given by
\begin{align}\label{MaxPotential_p5}
    U(B_{max}) 
    & = 8 (\alpha + 1)^2 \frac{M^2 H}{q-1} \left(\frac{\alpha M^2H}{(q-1)E}\right)^{2\alpha}.  
\end{align}

Rolling back the substitution \eqref{B} in \eqref{ParticleEnergy_B_p5}, gives
\begin{equation}\label{ParticleEnergy_p5}
    \mathcal{E}(t) = (\alpha+1 )^2 V^{2\alpha} \left[\frac12{ V_t^2} -{16 EV} +2M^2 H\right]. 
\end{equation}

Thus, to write the criteria (I) and (II) explicitly,
we plug in \eqref{MaxPotential_p5}, \eqref{E:B_max_p5} and ensure that $(2\alpha+1),H,M[u] \neq 0$, with $\alpha = \frac{q-5}8$, we get

\begin{align}\tag{A.1}\label{BupA1_p5}
V[u_0]^{\frac{q-5}4}\left[\frac14\frac{ V_t(0)^2}{M^2H} - 8 \frac{EV[u_0]}{M^2H} + 1\right] & < \frac{4}{q-1}\left(\frac{(q-5)}{8(q-1)} \frac{M^2H}{E}\right)^{\frac{q-5}4} \\
\tag{A.2}\label{BupA2_p5}
\qquad \mbox{with} \quad  V[u_0] & < \frac{(q-5)}{8(q-1)} \frac{M^2H}{E},
\end{align}      
\begin{align}\tag{B.1}\label{BupB1_p5} 
    V[u_0]^{\frac{q-5}4}\left[\frac14\frac{ V_t(0)^2}{M^2H} - 8 \frac{EV[u_0]}{M^2H} + 1\right] & > \frac{4}{q-1}\left(\frac{(q-5)}{8(q-1)} \frac{M^2H}{E}\right)^{\frac{q-5}4} \\
\quad\tag{B.2}\label{BupB2_p5}
\mbox{with} \quad    V[u_0]^\alpha V_t(0) &< 0,
\end{align}
and
{
\begin{align}\tag{C.1}\label{BupC1_p5} 
    V[u_0]^{\frac{q-5}4}\left[\frac14\frac{ V_t(0)^2}{M^2H} - 8 \frac{EV[u_0]}{M^2H} + 1\right] & = \frac{4}{q-1}\left(\frac{(q-5)}{8(q-1)} \frac{M^2H}{E}\right)^{\frac{q-5}4} \\
\quad\tag{C.2}\label{BupC2_p5}
\mbox{with} \quad  V[u_0] & < \frac{(q-5)}{8(q-1)} \frac{M^2H}{E} \; \mbox{and } \; V[u_0]^\alpha V_t(0) < 0.
\end{align}
}

To consolidate the results above into a straight-forward estimate, we consider the re-scaling $$B = B_{max} \,b(s),$$ where the time is rescaled as 
$$ 
s=\gamma t, \quad \gamma = \frac{E(q-1)2\sqrt{2}}{M\sqrt{\alpha(2\alpha+1)H}}. 
$$
Consequently, we apply this re-scaling  
into \eqref{ParticleEnergy_B_p5} and factor, yielding 
\begin{align*}
    \mathcal{E}(t)
    &=U(B_{max})\left[\frac{2\alpha}{(\alpha+1)^2}b_s^2 - 2\alpha b(s)^{\frac{2\alpha + 1}{\alpha + 1}} + (2\alpha+1)b(s)^\frac{2\alpha}{\alpha+1}\right].
\end{align*}

Setting {$\tilde{U}(b(s)) = -2\alpha b(s)^{\frac{2\alpha + 1}{\alpha + 1}} + (2\alpha+1)b(s)^\frac{2\alpha}{\alpha+1}$} and $\mathcal E(t) = U(B_{max}) \tilde{\mathcal E}(s)$, we have \eqref{ParticleEnergy} transform into
\begin{equation*}
    \tilde{\mathcal E}(s) = \frac{2\alpha}{(\alpha+1)^2}b_s^2 + \tilde{U}(b(s)).
\end{equation*}

Next, define $\tilde{V} = b^\frac{1}{\alpha + 1}$. Then we have $b_s = (\alpha+1)\tilde{V}^{\alpha}\tilde{V}_s$, and hence,
\begin{align*}
    \tilde{\mathcal E}(s) &= \frac{2\alpha}{(\alpha+1)^2}(\alpha+1)^2\tilde{V}^{2\alpha}\tilde{V}_s^{2} -2\alpha \tilde{V}^{2\alpha + 1} + (2\alpha+1)\tilde{V}^{2\alpha} 
    = {2}\alpha \tilde{V}^{2\alpha}(\tilde{V}_s^2 - \tilde{V} + \frac{2\alpha+1}{2\alpha}).
\end{align*}

For brevity, we introduce the function

\begin{equation}\label{CriteriaFunction}
f(X) = \sqrt{\frac{1}{2\alpha X^{2\alpha}} + X - \frac{2\alpha + 1}{2\alpha}}. 
\end{equation}

Then the conditions \eqref{BupA1_p5}-\eqref{BupA2_p5}, {\eqref{BupB1_p5}-\eqref{BupB2_p5}, and \eqref{BupC1_p5}-\eqref{BupC2_p5}} can be written as
\begin{center}
{\begin{varwidth}{\textwidth}
    \begin{enumerate}
        \item[(A)] $\tilde{\mathcal E}(0) < 1 \; \iff \; |\tilde{V}_s(0)| < f(\tilde{V}(0))$
        \item[(B)] $\tilde{\mathcal E}(0) > 1 \; \iff \; |\tilde{V}_s(0)| > f(\tilde{V}(0))$
        \item[(C)] $\tilde{\mathcal E}(0) = 1 \; \iff \; |\tilde{V}_s(0)| = f(\tilde{V}(0))$,
    \end{enumerate}
\end{varwidth}}
\end{center}
which is equivalent to 
\begin{center}
\begin{varwidth}{\textwidth}
    \begin{enumerate}
        \item $\tilde{V}(0) < 1$ and $-f(\tilde{V}(0)) < \tilde{V}_s(0) < f(\tilde{V}(0))$
        \item $\tilde{V}_s(0) {<} -f(\tilde{V}(0))$
        {\item $\tilde{V}(0) < 1$ and $\tilde{V}_s(0) =-f(\tilde{V}(0))$.}
    \end{enumerate}
\end{varwidth}
\end{center}
Combining the two cases into one, we obtain
\begin{equation*}
    \tilde{V}_s(0) < \begin{cases}
        f(\tilde{V}(0)) & \text{if} \; \tilde{V}(0) < 1 \\
        -f(\tilde{V}(0)) & \text{if} \; \tilde{V}(0) \geq 1.
    \end{cases}
\end{equation*}
Then, retracing our steps through the rescaling to solve for $V(t)$, we have 
\begin{equation}\label{E:V_rescaled_p5}
    V(t) = \frac{\alpha M^2H}{E(q-1)}\tilde{V}\left(\frac{2\sqrt{2}E(q-1)t}{M\sqrt{\alpha(2\alpha+1)H}}\right).
\end{equation}

Thus, the criterion for blow-up is
\begin{equation}\label{E:BUP_criteria_p5}
    \frac{V_t(0)}{2M\sqrt{H}} < \sqrt{\frac{q-5}{q-1}}g\left(\frac{8(q-1)E V[u_0]}{(q-5)M^2H}\right),
\end{equation}

where
\begin{equation}
g(X)=
    \begin{cases}
        f(X) & \text{if } 0<X< 1\\
        -f(X) & \text{if } X\geq1.
    \end{cases}
\end{equation}

\subsubsection{\underline{\bf General $p\neq 5$}} 

In this case (recalling $\epsilon_2>0$), from \eqref{Vdt2} and \eqref{E:Energy} we have
\begin{equation}\label{genpEnergy}
E[u] = \frac{1}{2}\|u_x\|_{L^2}^2 + \frac{|\epsilon_1|}{p+1}\|u\|_{L^{p+1}}^{p+1} - \frac{\epsilon_2}{q+1}\|u\|_{L^{q+1}}^{q+1},
\end{equation}
and 
\begin{equation}\label{genpVariancett}
V_{tt}[u] = 8\|u_x\|_{L^2}^2 + \frac{4|\epsilon_1|(p-1)}{p+1}\|u\|_{L^{p+1}}^{p+1} - \frac{4\epsilon_2(q-1)}{q+1}\|u\|_{L^{q+1}}^{q+1}.
\end{equation}
Re-scaling \eqref{genpEnergy} to a factor of $4(q-1)$ and solving for the $\|u\|_{L^{q+1}}$ term, we obtain
\begin{equation*}
     \frac{4(q-1)\epsilon_2}{q+1}\|u\|_{L^{q+1}}^{q+1}= 2(q-1)\| u_x\|_{L^2}^2 + \frac{4|\epsilon_1|(q-1)}{p+1}\|u\|_{L^{p+1}}^{p+1}-4(q-1)E[u].
\end{equation*}
Substituting the above into \eqref{genpVariancett} and simplifying, we obtain
\begin{equation*}
V_{tt}[u] = 2\left[5- q\right]\|u_x\|_{L^2}^2 - \frac{4|\epsilon_1|(q-p)}{p+1}\|u\|_{L^{p+1}}^{p+1} +4(q-1)E[u].
\end{equation*}
As in the case of $p = 5$ we utilize the inequalities \eqref{DRIneq} and \eqref{GradMassVar} and reorder appropriate terms to get
\begin{equation}\label{Gen_polynomial}
    VV_{tt} + \frac{(q-5)}{8}V_t^2 \leq  4(q-1)EV - \frac12{(q-5)}M^2 - \frac{4|\epsilon_1|(q-p)}{p+1}\left[\frac{M^{\frac{p-1}{4} + \frac{p+1}{2}}}{V^\frac{p-5}{4}C^*}\right].
\end{equation}
From here we split our reasoning into two cases: (\textbf{Case I}) keeping  the term $- \frac12{(q-5)}M^2$, (\textbf{Case II}) dropping this term by bounding it as $- \frac12{(q-5)}M^2 < 0$.
\medskip

\noindent \underline{\textbf{Case I:}} Multiplying \eqref{Gen_polynomial} by $(\alpha+1)V^{\alpha - 1}$, where $\alpha = \frac{q-5}{8}$,
{\small
\begin{align*}
(\alpha+1) V^{\alpha-1}    [VV_{tt} + \frac{(q-5)}{8}V_t^2] &\leq  
(\alpha+1) \left(4(q-1)E V^\alpha - \frac12 {(q-5) M^2} V^{\alpha-1} - \frac{4|\epsilon_1|(q-p)}{p+1} \frac{M^{\frac{p-1}{4} + \frac{p+1}{2}}}{C^*} V^{\alpha - \frac{p-1}{4}}\right),
\end{align*}
}
we obtain a {sum of powers} $\alpha$ in $V$, with the second term having larger power of $V$ than the third one if $p>5$, and vice versa, the third term having larger power than the second, if $p<5$. Then using the substitution as in \eqref{B} we obtain 
{
\begin{align*}
    B_{tt} \leq (\alpha+1) \left( 4(q-1)EB^{\frac{\alpha}{\alpha + 1}}  -
    \frac12 (q-5) M^2 B^{\frac{\alpha-1}{\alpha+1}} 
    - \frac{4|\epsilon_1|(q-p)}{(p+1)C^*}M^{\frac{p-1}{4} + \frac{p+1}{2}}B^{\frac{\alpha - \frac{p - 1}4}{(\alpha+1)}} \right).
\end{align*}
}
The argument follows from the physical motivation outlined in the case for $p=5$. Since it is identical, we omit the motivation for the technique and only apply the methods and estimates here. Thus, we define 

{
\begin{equation}\label{Long_potential_deriv}
    -\frac{\partial U}{\partial B} = 
    (\alpha+1) \left( 4(q-1)EB^{\frac{\alpha}{\alpha + 1}}  -
    \frac12 (q-5) M^2 B^{\frac{\alpha-1}{\alpha+1}} 
    - \frac{4|\epsilon_1|(q-p)}{(p+1)C^*}M^{\frac{p-1}{4} + \frac{p+1}{2}}B^{\frac{\alpha - \frac{p - 1}4}{(\alpha+1)}} \right),
\end{equation}
then integrating with respect to $B$,
{
\begin{equation}\label{Potential_genp_long}
    U(B) =  - (\alpha+1)^2
    \left( 16 EB^{\frac{2\alpha+1}{\alpha + 1}}
    {-} 2 M^2 B^{\frac{2\alpha}{\alpha + 1}}
    {-
     \frac{16 |\epsilon_1|}{(p+1){C^*}}{M^{\frac{p-1}{4}+\frac{p+1}{2}}} B^{\frac{q-p}{4(\alpha+1)}}} \right).
\end{equation}
}

Our goal now becomes finding the roots of {\cancel{the polynomial}} \eqref{Long_potential_deriv}. In the $p=5$ case this was straightforward. However, notice that without specific $p$ or $q$ we cannot express $B_{max}$ explicitly. Instead, by the intermediate value theorem, we know that it must {lie in a range} for two cases $p > 5$ and $1 < p < 5$ leading to the proposition.

\begin{proposition}\label{ZerosOfGenP}
Define the maximum between $\mathcal{L}$ and $\mathcal{M}$ as the lower bound of $B_{max}$ such that $$\mathcal{M} :=\left(\frac{|\epsilon_1|(q-p)M^{\frac{p-1}{4} + \frac{p+1}{2}}}{(p+1)(q-1)C^*E}\right)^\frac{4(\alpha+1)}{(p-1)} \mbox{ and }\; \mathcal{L}:= \left(\frac{(q-5) M^2}{8(q-1)E}\right)^{\alpha+1}.$$
Then, if $ 1 < p < 5$  
$$\max\{\mathcal{L},\mathcal{M}\} < B_{max} < \max\left\{\mathcal{M}\left(\Big(\frac{\mathcal{L}}{\mathcal{M}}\Big)^{\frac{1}{\alpha + 1}}   
    + 1\right)^{\frac{4(\alpha + 1)}{p-1}},\left( \mathcal{L}^{\frac{p-1}{4(\alpha+1)}} 
    + \mathcal{M}^{\frac{p-1}{4(\alpha+1)} }\right)^{\frac{4(\alpha + 1)}{p-1}}\right\}.$$
If $p > 5$, 
$$\max\{\mathcal{L},\mathcal{M}\} < B_{max} {\leq} \max\left\{2^{\frac{4(\alpha+1)}{p-1}}\mathcal{M},2^{\alpha+1}\mathcal{L}\right\}.$$
\end{proposition}

\begin{proof}[Proof of Proposition \ref{ZerosOfGenP}]
    Define the function,
    \begin{equation}\label{factored}
        f(B) = -4(q-1)EB^{\frac{p-1}{4(\alpha + 1)}}  +
    \frac12 (q-5) M^2 B^{\frac{p-5}{4(\alpha+1)}} 
    + \frac{4|\epsilon_1|(q-p)}{(p+1)C^*}M^{\frac{p-1}{4} + \frac{p+1}{2}}.
    \end{equation}
    Setting \eqref{factored} equal to $0$ and reorganizing we obtain,
    \begin{equation}\label{reorganize}
        4(q-1)EB^{\frac{p-1}{4(\alpha + 1)}}  =
    \frac12 (q-5) M^2 B^{\frac{p-5}{4(\alpha+1)}} 
    + \frac{4|\epsilon_1|(q-p)}{(p+1)C^*}M^{\frac{p-1}{4} + \frac{p+1}{2}}.
    \end{equation}
    The lower bound is straightforward, note 
    \begin{equation*}
        4(q-1)EB^{\frac{p-1}{4(\alpha + 1)}}  >
    \frac12 (q-5) M^2 B^{\frac{p-5}{4(\alpha+1)}},
    \end{equation*}
    thus
    \begin{equation*}
        B_{max}  > \left(
     \frac{(q-5) M^2}{8(q-1)E}\right)^{\alpha+1} =: \mathcal{L}.
    \end{equation*}
    We could also choose,
    \begin{equation*}
        4(q-1)EB^{\frac{p-1}{4(\alpha + 1)}}  >
     \frac{4|\epsilon_1|(q-p)}{(p+1)C^*}M^{\frac{p-1}{4} + \frac{p+1}{2}},
    \end{equation*}
    thus
    \begin{equation*}
        B_{max}  > \left(
     \frac{|\epsilon_1|(q-p)M^{\frac{p-1}{4} + \frac{p+1}{2}}}{(p+1)(q-1)C^*E}\right)^{\frac{4(\alpha + 1)}{p-1}} =: \mathcal{M}.
    \end{equation*}
    Then the optimal lower bound is chosen as the largest between $\mathcal{M}$ and $\mathcal{L}$.
    
    To proceed with the upper bound we split into two cases, one where $1 < p < 5$ and the other when $ p > 5$. 

\textbf{Case 1} ($1 < p < 5$):

    Notice that $B^{\frac{p-5}{4(\alpha + 1)}}$ is a decreasing function, then if $\mathcal{M} < \mathcal{L}$ \eqref{reorganize} can be expressed as 
    \begin{equation}
        4(q-1)EB^{\frac{p-1}{4(\alpha + 1)}}  <
    \frac12 (q-5) M^2 \mathcal{M}^{\frac{p-5}{4(\alpha+1)}} 
    + \frac{4|\epsilon_1|(q-p)M^{\frac{p-1}{4} + \frac{p+1}{2}}}{(p+1)C^*},
    \end{equation}
    and
    \begin{align*}
        B_{max}  &<
    \left(\frac{(q-5) M^2}{8(q-1)E}  \mathcal{M}^{\frac{p-5}{4(\alpha+1)}} 
    + \frac{|\epsilon_1|(q-p)M^{\frac{p-1}{4} + \frac{p+1}{2}}}{(p+1)(q-1)C^*E}\right)^{\frac{4(\alpha + 1)}{p-1}}\\
    &=\mathcal{M}\left(\Big(\frac{\mathcal{L}}{\mathcal{M}}\Big)^{\frac{1}{\alpha + 1}}   
    + 1\right)^{\frac{4(\alpha + 1)}{p-1}}.
    \end{align*}
    However, if $\mathcal{L} < \mathcal{M}$,
    \begin{align*}
        B_{max}  &<
    \left(\frac{(q-5) M^2}{8(q-1)E}  \mathcal{L}^{\frac{p-5}{4(\alpha+1)}} 
    + \frac{|\epsilon_1|(q-p)M^{\frac{p-1}{4} + \frac{p+1}{2}}}{(p+1)(q-1)C^*E}\right)^{\frac{4(\alpha + 1)}{p-1}}\\
    &= \left( \mathcal{L}^{\frac{p-1}{4(\alpha+1)}} 
    + \mathcal{M}^{\frac{p-1}{4(\alpha+1)} }\right)^{\frac{4(\alpha + 1)}{p-1}}. 
    \end{align*}
    {C}hoosing the largest between the two yields the upper bound.
    {\begin{remark}
        Note, if $\mathcal{L} = \mathcal{M}$, we have $B_{max} < 2^{\frac{4(\alpha+1)}{p-1}}\mathcal{M}.$
    \end{remark}}
    \textbf{Case 2} ($5 < p < q$):
    For this case we use the maximum between the two terms on the right hand side of \eqref{reorganize} to obtain,

    \begin{equation*}
        4(q-1)EB^{\frac{p-1}{4(\alpha + 1)}}  {\leq}
    2\max\left\{\frac12 (q-5) M^2 B^{\frac{p-5}{4(\alpha+1)}} 
    , \frac{4|\epsilon_1|(q-p)M^{\frac{p-1}{4} + \frac{p+1}{2}}}{(p+1)C^*}\right\}.
    \end{equation*}
    We start by assuming the first term on the right hand side is the maximum, then

    \begin{align*}
        B_{max} {\leq}
    \left(\frac{(q-5) M^2}{4(q-1)E}\right)^{(\alpha + 1)}  = 2^{\alpha+1}\mathcal{L}.
    \end{align*}
    If instead we assume the second term on the right hand side is the maximum, we have
    \begin{align*}
        B_{max} {\leq}
    \left(\frac{2|\epsilon_1|(q-p)M^{\frac{p-1}{4} + \frac{p+1}{2}}}{(p+1)(q-1)C^*E}\right)^{\frac{4(\alpha + 1)}{p-1}}  = 2^{\frac{4(\alpha + 1)}{p-1}}\mathcal{M}.
    \end{align*}
    Taking the maximum between the two choices yields our final upper bound. Completing the proof for all cases.
\end{proof}
}

We have successfully provided a suitable range for the critical point $B_{max}$, however, to proceed with the proof it is needed to have an explicitly calculated $B_{max}$. We encourage future investigations in this range of $B_{max}$ but to stay within the scope of this paper, we continue with a different approach in \textbf{Case II.}
\medskip

\noindent\underline{\textbf{Case II:}}
For this case, notice that $- \frac{(q-5)}{2}M^2$ is constant and helps the RHS of the inequality stay negative. Thus, we bound $- \frac{(q-5)}{2}M^2 < 0$ knowing that the concluding criteria will not be as sharp. Hence,
\begin{align*}
    VV_{tt} + \frac{(q-5)}{8}V_t^2 &\leq  4(q-1)EV - \frac{4|\epsilon_1|(q-p)}{p+1}\left[\frac{M^{\frac{p-1}{4} + \frac{p+1}{2}}}{V^\frac{p-5}{4}C^*}\right].
\end{align*}

Then, using the same rescaling ($\alpha = \frac{q-5}{8}$ again) and substitutions \eqref{B},  
as in the $p = 5$ case, we obtain, after solving for $B_{tt}$

\begin{align*}
    B_{tt}  &\leq  4(\alpha + 1)(q-1)EB^{\frac{\alpha}{\alpha + 1}}  - \frac{4|\epsilon_1|(q-p)(\alpha+1)}{(p+1)C^*}M^{\frac{p-1}{4} + \frac{p+1}{2}}B^{\frac{4\alpha -p + 1}{4(\alpha+1)}}.
\end{align*}

As before we omit the motivation for the technique and only apply the methods and estimates here. Thus, we define 
\begin{equation*}
    -\frac{\partial U}{\partial B} = 4(\alpha + 1)(q-1)EB^{\frac{\alpha}{\alpha + 1}}  - \frac{4|\epsilon_1|(q-p)(\alpha+1)}{(p+1)C^*}M^{\frac{p-1}{4} + \frac{p+1}{2}}B^{\frac{4\alpha -p + 1}{4(\alpha+1)}},
\end{equation*}

and 
\begin{equation}\label{Potential_genp}
    U(B) = \frac{16|\epsilon_1|(\alpha+1)^2}{(p+1)C^*}M^{\frac{p-1}{4} + \frac{p+1}{2}}B^{\frac{q-p}{4(\alpha+1)}} - 16(\alpha+1)^2
    EB^{\frac{2\alpha+1}{\alpha + 1}}.
\end{equation}

In this case, the energy of the particle as expressed in \eqref{ParticleEnergy} is given by
\begin{equation}\label{ParticleEnergy_B_genp}
    \mathcal{E}(t) = \frac{B_t^2}{2} +\frac{16|\epsilon_1|(\alpha+1)^2}{(p+1)C^*}M^{\frac{p-1}{4} + \frac{p+1}{2}}B^{\frac{q-p}{4(\alpha+1)}} - 16(\alpha+1)^2
    EB^{\frac{q-1}{4(\alpha + 1)}},
\end{equation}

Next, we find the maximum for \eqref{Potential_genp} by setting $\frac{\partial U}{\partial B} = 0$ and solving for $B$ to obtain $B_{max}$, which corresponds to the maximum as in Figure \ref{F:PotentialEnergyGraph},
\begin{equation*}
    B_{max} = \left( \frac{|\epsilon_1|(q-p)M^{\frac{p-1}{4} + \frac{p+1}{2}}}{C^*(q-1)(p+1)E}\right)^{\frac{4(\alpha+1)}{p-1}}.
\end{equation*}
We remark that with the adjustment in this case, we are able to obtain an explicit $B_{max}$. Then 
\begin{equation*}
    U(B_{max}) = \frac{16 (\alpha+1)^2(p-1)E}{(q-p)}
    \left( \frac{|\epsilon_1|(q-p)M^{\frac{p-1}{4} + \frac{p+1}{2}}}{(q-1)(p+1)C^* E}\right)^{\frac{q-1}{p-1}}.
\end{equation*}

The energy of the particle $B$ in \eqref{ParticleEnergy_B_genp} after reversing the substitution \eqref{B} and simplifying is
\begin{align*}
    \mathcal{E}(t) = \frac12(\alpha+1)^2V^{2\alpha}V_t^2 
    + 16(\alpha+1)^2 E V^{\frac{q-1}4} \left[\frac{|\epsilon_1| M^{\frac{p-1}{4} + \frac{p+1}{2}}}{(p+1)C^*E} V^{\frac{1-p}{4}}
    -  1 \right].
\end{align*}

Thus, our explicit criterion is
 {\small
    \begin{align}
        \frac1{2}V[u_0]^{2\alpha}V_t(0)^2 + 16 E V{[u_0]}^{\frac{q-1}4}\left[\frac{|\epsilon_1|M^{\frac{p-1}{4} + \frac{p+1}{2}}}{(p+1)C^*E}V[u_0]^{\frac{1-p}{4}} - {1}\right] &< \frac{16(p-1)E}{(q-p)}\left( \frac{|\epsilon_1|(q-p)M^{\frac{p-1}{4} + \frac{p+1}{2}}}{(q-1)(p+1)C^*E}\right)^{\frac{q-1}{p-1}}, \tag{D.1}\label{BupA1_genp}\\
    \tag{D.2}\label{BupA2_genp}
  \mbox{with} \quad      V[u_0] &< \left( \frac{|\epsilon_1|(q-p)M^{\frac{p-1}{4} + \frac{p+1}{2}}}{(q-1)(p+1)C^*E}\right)^{\frac{4}{p-1}}
    \end{align} 
    }
{\small    
\begin{align}
        \frac{1}{2}V[u_0]^{2\alpha}V_t(0)^2 + 16 EV{[u_0]}^{\frac{q-1}{4}}\left[{\frac{|\epsilon_1|M^{\frac{p-1}{4} + \frac{p+1}{2}}}{(p+1)C^*E}V[u_0]^{\frac{1-p}{4}}} - 1\right] &> \frac{16(p-1)E}{(q-p)}\left( \frac{|\epsilon_1|(q-p)M^{\frac{p-1}{4} + \frac{p+1}{2}}}{(q-1)(p+1)C^*E}\right)^{\frac{q-1}{p-1}}\tag{E.1}\label{BupB1_genp} \\
    \tag{E.2}\label{BupB2_genp}
 \mbox{with} \quad    V{[u_0]^\alpha} V_t{(0)} &< 0,
    \end{align}
}
and
{\small    
\begin{align}
        \frac{1}{2}V[u_0]^{2\alpha}V_t(0)^2 + 16 EV[u_0]^{\frac{q-1}{4}}\left[\frac{|\epsilon_1|M^{\frac{p-1}{4} + \frac{p+1}{2}}}{(p+1)C^*E}V[u_0]^{\frac{1-p}{4}} - 1\right] &= \frac{16(p-1)E}{(q-p)}\left( \frac{|\epsilon_1|(q-p)M^{\frac{p-1}{4} + \frac{p+1}{2}}}{(q-1)(p+1)C^*E}\right)^{\frac{q-1}{p-1}}\tag{F.1}\label{BupC1_genp} \\
    \tag{F.2}\label{BupC2_genp}
 \mbox{with} \quad V[u_0] < \left( \frac{|\epsilon_1|(q-p)M^{\frac{p-1}{4} + \frac{p+1}{2}}}{(q-1)(p+1)C^*E}\right)^{\frac{4}{p-1}} &\mbox{ and }    V[u_0]^\alpha V_t{(0)} < 0.
    \end{align}
}

To encapsulate and simplify the criterion, similar to the case for $p=5$, we make the substitution 
\begin{equation*}
B = B_{max}\, b(s) \quad \mbox{with} \quad s = \gamma t, \quad \mbox{where} \quad  
    { \gamma = 4\sqrt{2E} \left[\frac{C^*(q-1)(p+1)E}{|\epsilon_1|(q-p)M^{\frac{p-1}{4} + \frac{p+1}{2}}}\right]^{\frac{2}{p-1}}}
\end{equation*}
Thus, after some simplification \eqref{ParticleEnergy_B_genp} is 
{
\begin{equation*}
    \mathcal{E}(t) = U(B_{max})\left[\frac{64(q-p)}{(p-1)(q+3)^2}b_s^2 + \frac{q-1}{p-1}b^{\frac{2(q-p)}{q+3}} - \frac{q-p}{p-1}b^{\frac{2(q-1)}{q+3}}\right].
\end{equation*}
}
Define
$\tilde{U}(b) = \frac{q-1}{p-1}b^{\frac{2(q-p)}{q+3}} - \frac{q-p}{p-1}b^{\frac{2(q-1)}{q+3}}${,}then the energy for the re-scaled particle $b$ is
{
\begin{equation*}
    \tilde{\mathcal{E}}(s) = \frac{64(q-p)}{(p-1)(q+3)^2}b_s^2 + \tilde{U}(b).
\end{equation*}
}
Next, let $\tilde{V} = b^\frac{1}{\alpha+1}$ and $b_s  = (\alpha+1)\tilde{V}^\alpha \tilde{V}_s$, then
{
\begin{equation*}
    \tilde{\mathcal{E}}(s) = \frac{q-p}{(p-1)}\tilde{V}^{\frac{q-5}{4}}\left(\tilde{V}_s^2 - \tilde{V} + \frac{q-1}{q-p}\tilde{V}^{\frac{5-p}{4}} \right).
\end{equation*}
}

Define the function 
\begin{equation*}
    f(X) = \sqrt{\frac{p-1}{q-p}X^{-\frac{q-5}4} + X - \frac{q-1}{q-p}X^{\frac{5-p}{4}}}.
\end{equation*}

Then the criteria \eqref{BupA1_genp}-\eqref{BupA2_genp}{, \eqref{BupB1_genp}-\eqref{BupB2_genp}, and \eqref{BupC1_genp}-\eqref{BupC2_genp}} are formulated as follows:
\begin{center}
{\begin{varwidth}{\textwidth}
    \begin{enumerate}
        \item[(A)] $\tilde{\mathcal E}(0) < 1 \; \iff \; |\tilde{V}_s(0)| < f(\tilde{V}(0))$
        \item[(B)] $\tilde{\mathcal E}(0) > 1 \; \iff \; |\tilde{V}_s(0)| > f(\tilde{V}(0))$
        \item[(C)] $\tilde{\mathcal E}(0) = 1 \; \iff \; |\tilde{V}_s(0)| = f(\tilde{V}(0))$,
    \end{enumerate}
\end{varwidth}}
\end{center}
which is equivalent to 
\begin{center}
\begin{varwidth}{\textwidth}
    \begin{enumerate}
        \item $\tilde{V}(0) < 1$ and $-f(\tilde{V}(0)) < \tilde{V}_s(0) < f(\tilde{V}(0))$
        \item $\tilde{V}_s(0) {<} -f(\tilde{V}(0))$,
        {\item $\tilde{V}(0) < 1$ and $\tilde{V}_s(0) =-f(\tilde{V}(0))$.}
    \end{enumerate}
\end{varwidth}
\end{center}
and can be expressed as
\begin{equation*}
    \tilde{V}_s(0) < \begin{cases}
        f(\tilde{V}(0)) & \text{if} \; \tilde{V}(0) < 1 \\
        -f(\tilde{V}(0)) & \text{if} \; \tilde{V}(0) \geq 1.
    \end{cases}
\end{equation*}

Working backwards from our substitutions, we obtain

{
\begin{equation*}
    V = \left( \frac{|\epsilon_1|(q-p)M^{\frac{p-1}{4} + \frac{p+1}{2}}}{(p+1)(q-1)C^*E}\right)^{\frac{4}{p-1}}\tilde{V}\left( 4\sqrt{2E}\left( \frac{|\epsilon_1|(q-p)M^{\frac{p-1}{4} + \frac{p+1}{2}}}{(p+1)(q-1)C^*E}\right)^{\frac{-2}{p-1}}t \right),
\end{equation*}
}
and 

{
\begin{equation*}
    V_t = 4\sqrt{2}\left( \frac{|\epsilon_1|(q-p)}{(p+1)(q-1)C^*}\right)^{\frac{2}{p-1}}{M^{\frac{3p+1}{2(p-1)}}}E^{\frac{p-5}{2(p-1)}}\tilde{V}_s\left( s \right),
\end{equation*}
}

Therefore, our final criterion is
{
\begin{equation*}
    \frac{V_t(0)}{\left( \frac{|\epsilon_1|(q-p)}{(p+1)(q-1)C^*}\right)^{\frac{2}{p-1}}{M^{\frac{3p+1}{2(p-1)}}}} < 4\sqrt{2E^{\frac{p-5}{p-1}}}g\left( \left( \frac{(p+1)(q-1)C^*E}{|\epsilon_1|(q-p)M^{\frac{p-1}{4} + \frac{p+1}{2}}}\right)^{\frac{4}{p-1}}V[u_0] \right),
\end{equation*}
}

where
\begin{equation}
g(X)=
    \begin{cases}
        f(X) & \text{if } 0<X< 1\\
        -f(X) & \text{if } X\geq1
    \end{cases}
\end{equation}
which {concludes} the proof of Theorem \ref{DR_Bup}.

\begin{remark}
    It would be interesting to see how this method can be applied to the case of triple nonlinearities. It may be possible with different signs in the coefficients to use techniques from the proof of Theorem \ref{NegE_Bup_3term} and apply them to this method.
\end{remark}

\section{Examples}\label{S:Examples}
We conclude this paper with the section that focuses on several explicit examples that satisfy the negative energy requirement with a quadratic phase and the positive energy blow-up criterion.

\subsection{Negative energy examples}
\subsubsection{Gaussian}\label{S:NegE_Gaussian}
In \eqref{E:InitialDataNegE}, let $v_0 = e^{i\theta x^2}Ae^{-\frac{x^2}{2}}$ with $A>0$. Then, $\partial_x(v_0) = (2i\theta x - x)v_0$. Inserting into \eqref{E:InitialDataNegE} we obtain
\begin{equation*}
    \frac{A^2|b|}{4}\int_{-\infty}^\infty|x e^{\frac{-x^2}{2}}|^2dx < \Im \left(A^2\int_{-\infty}^{\infty}(2i\theta - 1)|xe^{\frac{-x^2}{2}}|^2dx\right), 
\end{equation*}
where we have used the fact that $|e^{i\theta}| = 1$ and $v_0 \overline{v_0} = |v_0|^2$. Canceling $A^2$ from both sides and taking the imaginary part, yields 
\begin{equation}
    \frac{|b|}{4}\int_{-\infty}^\infty|x e^{\frac{-x^2}{2}}|^2dx < 2\theta\int_{-\infty}^{\infty}|xe^{\frac{-x^2}{2}}|^2dx. 
\end{equation}
Thus, we are left with the condition $\frac{|b|}{8} < \theta.$

We next verify that $E[v_0] < 0$ under specific conditions on the parameters $\theta$ and $A$:
\begin{align*}
    E[v_0] &= \frac{1}{2}\int_{-\infty}^\infty|\partial_x v_0|^2dx -\sum_{k=0}^{\infty}\frac{d_k}{\alpha_k+2} \int_{-\infty}^\infty|v_0|^{\alpha_k + 2} \, dx \\
    &= (4\theta^2 + 1)A^2\int_{-\infty}^\infty x^2e^{-x^2}dx -\sum_{k=0}^{\infty}\frac{d_k}{\alpha_k+2} \int_{-\infty}^\infty|Ae^{-\frac{x^2}{2}}|^{\alpha_k + 2} dx \\
    &= \frac{\sqrt{\pi}(4\theta^2 + 1)A^2}{2} - \sqrt{2\pi} \sum_{k=0}^{\infty}\frac{d_k A^{\alpha_k + 2}}{(\alpha_k+2)^\frac{3}{2}}. 
\end{align*}

Our goal is then to establish conditions on $\theta$, $A$, $\alpha_k$ and $d_k$ such that the following inequality holds
\begin{equation}\label{SubCondition2}
    (4\theta^2 + 1) <  2\sqrt2\sum_{k=0}^{\infty}\frac{d_k A^{\alpha_k}}{(\alpha_k+2)^\frac{3}{2}}. 
\end{equation}

Depending on the coefficients $d_k$  and $\alpha_k$, condition \eqref{SubCondition2} can be satisfied by choosing $\theta$ and $A$ appropriately.

\textbf{Three Terms Case:}
One may choose any of the combinations presented in Theorem \ref{NegE_Bup_3term} to obtain an example of initial data that blow up in finite time. For simplicity we choose an example that matches \textbf{Case 1}, say, $\lambda_1 = \lambda_2 = 1, \lambda_3 = -1$ and $\alpha_1 = 3, \alpha_2 = 5,  \alpha_3 = 7$. Then \eqref{SubCondition2} becomes
\begin{equation}\label{SubCondition2_combined}
    (4\theta^2 + 1) <  2\sqrt2A^3\left(\frac{ -1}{5^{\frac{3}{2}}} - \frac{ A^2}{7^{\frac{3}{2}}} + \frac{ A^4}{9^{\frac{3}{2}}}\right). 
\end{equation}
The range of acceptable $\theta$ and $A$ is given by Figure \ref{fig:combined_specific_example}.
\begin{figure}[htb!]
    \centering
    \begin{minipage}{.4\textwidth}
        \centering    \includegraphics[width=0.6\linewidth]{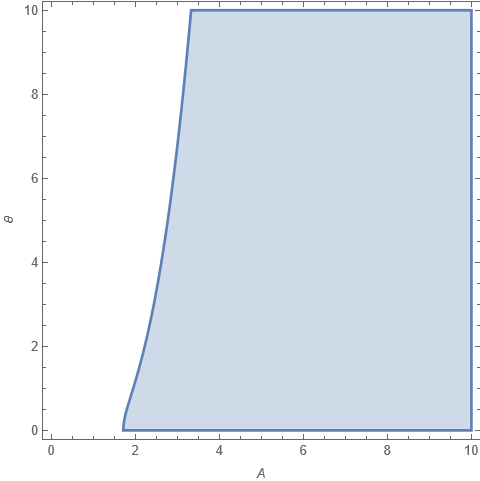}
        \captionof{figure}{\small The inequality \eqref{SubCondition2_combined} holds inside the blue area, with $\theta \in (0,10)$ and $A \in (0,10)$.}
        \label{fig:combined_specific_example}
    \end{minipage}%
    \begin{minipage}{.5\textwidth}
        \centering    \includegraphics[width=0.7\linewidth]{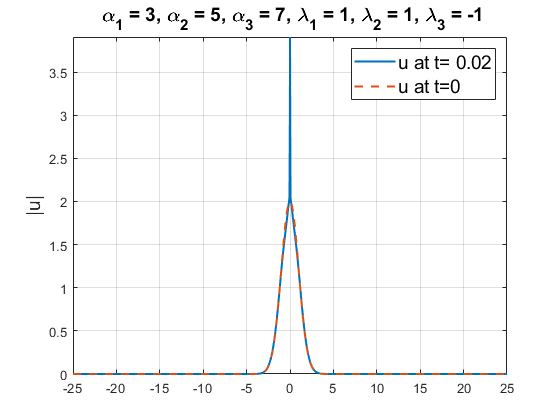}
        \captionof{figure}{\small Blow-up profile at $t=0.02$ with $A = 2$ and $\theta = 0.1$, where the initial condition is of the form $u_0 = e^{i\theta x^2}Ae^{-\frac{x^2}{2}}.$ }
        \label{fig:Blowup_three_term}
    \end{minipage}
\end{figure}
We mention that several cases of combined NLS are studied in \cite{RRR2024} (up to double nonlinearity). This case is not covered in \cite{RRR2024}; we choose $\theta = 0.1$ and $A = 2$ to show a specific simulation, given in Figure \ref{fig:Blowup_three_term}. It suggests a finite time blow-up, confirming the criterion in Theorem \ref{NegE_Bup_3term}. 

\smallskip

\textbf{Exponential Case:}
One can choose the specific case of the exponential, where $d_k = \frac{1}{k!}$ and $\alpha_k = k$. Which corresponds to the criteria of Theorem \ref{NegE_Bup_Exp}. Then \eqref{SubCondition2} becomes
\begin{equation}\label{SubCondition2_exp}
    (4\theta^2 + 1) <  2\sqrt2\sum_{k=0}^{\infty}\frac{ A^{k}}{k!(k+2)^\frac{3}{2}}. 
\end{equation}

An estimated range of $\theta$ and $A$ can be computed numerically, for example, taking the sum to $N = 100$ we obtain the range shown in Figure \ref{fig:gaussian_specific_example}.
\begin{figure}[htb!]
    \centering
    \begin{minipage}{.5\textwidth}
        \centering    \includegraphics[width=0.5\linewidth]{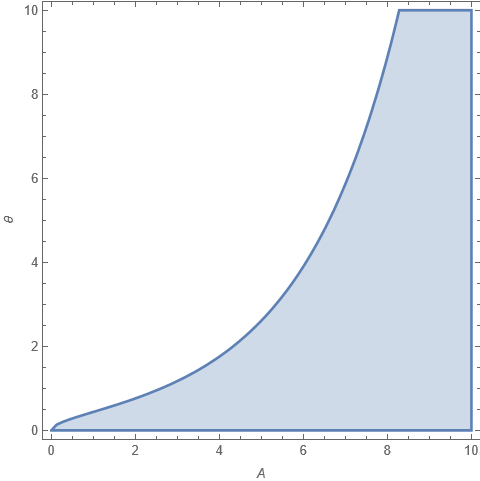}
        \captionof{figure}{\footnotesize The inequality \eqref{SubCondition2_exp} holds inside the blue area, with $\theta \in (0,10)$ and $A \in (0,10)$.}
        \label{fig:gaussian_specific_example}
    \end{minipage}%
    \begin{minipage}{.5\textwidth}
        \centering    \includegraphics[width=0.7\linewidth]{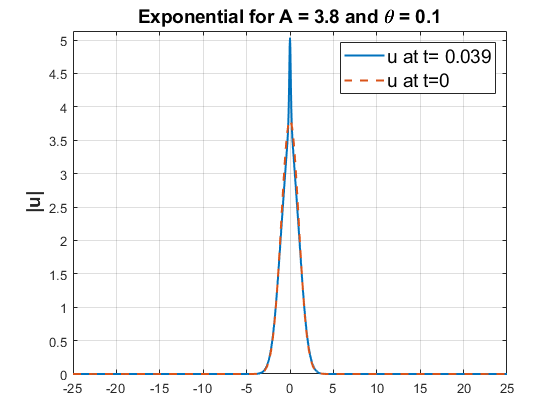}
        \captionof{figure}{\footnotesize Blow-up profile with $A = 3.8$ and $\theta = 0.1$, where the initial condition is of the form $u_0 = e^{i\theta x^2}Ae^{-\frac{x^2}{2}}.$ }
        \label{fig:Blowup_exp}
    \end{minipage}
\end{figure}
\newpage
We take $\theta = 0.1$ and $A = 3.8$ to show a specific simulation of the time evolution of the initial condition $u_0 = e^{i\theta x^2}Ae^{-\frac{x^2}{2}}$ in the exponential NLS, see Figure \ref{fig:Blowup_exp}. It suggests a finite time blow-up, confirming the criterion in Theorem \ref{NegE_Bup_Exp}.

\subsubsection{Weighted example}\label{S:NegE_Weighted}
We now investigate an example that lives inside a particular weighted Sobolev space $\mathcal X$ (see definition in \cite{RRR2024}) with data that have polynomial decay at infinity. In \eqref{E:InitialDataNegE}, choose 
$$v_0 = \frac{A}{\langle x \rangle^n}e^{\frac{-i\theta}{\langle x \rangle}}$$ 
with $\partial_x(v_0) = \frac{2Ai\theta x }{\langle x \rangle^{n+3}} e^{\frac{-i\theta}{\langle x \rangle}} - \frac{nAx }{\langle x \rangle^{n+2}}e^{\frac{-i\theta}{\langle x \rangle}}$, where $\langle x \rangle = (1+x^2)^{\frac{1}{2}}$. We go through a familiar procedure and write

\begin{align*}
     \Im \int_{-\infty}^\infty x \partial_x (v_0(x))\overline{v_0(x)}dx &= \Im \int_{-\infty}^\infty x \left(\frac{2Ai\theta x e^{\frac{-i\theta}{\langle x \rangle}}}{\langle x \rangle^{n+3}} - \frac{nA x e^{\frac{-i\theta}{\langle x \rangle}}}{\langle x \rangle^{n+2}}\right) \left(\frac{Ae^{\frac{i\theta}{\langle x \rangle}}}{\langle x \rangle^n}\right)dx \\
     &= \Im A^2 \int_{-\infty}^\infty x \left(\frac{2i\theta x }{\langle x \rangle^{2n+3}} - \frac{n x }{\langle x \rangle^{2n+2}}\right)dx \\
     &= A^2\int_{-\infty}^\infty x \left(\frac{2\theta x }{\langle x \rangle^{2n+3}}\right)dx = 2\theta A^2 \int_{-\infty}^\infty  \frac{ x^2 }{\langle x \rangle^{2n+3}}dx,
\end{align*}
and
\begin{align*}
    \frac{|b|}{4}\int_{-\infty}^\infty|x v_0(x)|^2dx &= \frac{|b|A^2}{4}\int_{-\infty}^\infty\left| \frac{xe^{\frac{-i\theta}{\langle x \rangle}}}{\langle x \rangle^n}\right|^2dx = \frac{|b|A^2}{4}\int_{-\infty}^\infty \frac{x^2}{\langle x \rangle^{2n}}dx.
\end{align*}
For \eqref{E:NegE_requirement} we require $n > \frac{3}{2}$.
Then,
\begin{equation*}
    \frac{|b|}{4}\int_{-\infty}^\infty \frac{x^2}{\langle x \rangle^{2n}}dx < 2\theta \int_{-\infty}^\infty  \frac{ x^2 }{\langle x \rangle^{2n+3}}dx,
\end{equation*}
which can be written as
\begin{equation*}
    \frac{|b|\sqrt{\pi}\Gamma\left(n - \frac{3}{2}\right)}{8\Gamma\left(n\right)} <   \frac{\theta\sqrt{\pi}\Gamma\left(n\right)}{\Gamma\left(\frac{3}{2} + n\right)}.
\end{equation*}
Thus, we pick $\theta$ such that
\begin{equation*}
    \frac{|b|\Gamma\left(\frac{3}{2} + n\right)\Gamma\left(n - \frac{3}{2}\right)}{8\Gamma\left(n\right)^2} <   \theta.
\end{equation*}
Notice when $n=2$
\begin{equation*}
    \frac{|b|\pi}{8} < \frac{\theta}{2} \implies \theta > \frac{|b|\pi}{4}.
\end{equation*}
Now we would like to verify that the energy $E[v_0]$ is in fact negative. Recall
\begin{align*}
    E[v_0] = \frac{1}{2}\int_{-\infty}^\infty|\partial_x v_0|^2dx -\sum_{k=0}^{\infty}\frac{d_k}{\alpha_k+2} \int_{-\infty}^\infty|v_0|^{\alpha_k + 2} \, dx.
\end{align*}
Looking at it term by term, we get
\begin{align*}
    \frac{A^2}{2}\int_{-\infty}^\infty\left|\frac{2i\theta x e^{\frac{-i\theta}{\langle x \rangle}}}{\langle x \rangle^{n+3}} - \frac{n x e^{\frac{-i\theta}{\langle x \rangle}}}{\langle x \rangle^{n+2}}\right|^2dx &= \frac{A^2}{2}\int_{-\infty}^\infty\frac{x^2}{\langle x \rangle^{2n+4}}\left(\frac{4\theta^2}{\langle x \rangle^2} + n^2\right)dx \\
    &= 2\theta^2A^2 \int_{-\infty}^\infty \frac{x^2}{\langle x \rangle^{2n+6}}dx + \frac{n^2A^2}{2} \int_{-\infty}^\infty \frac{x^2}{\langle x \rangle^{2n+4}}dx
\end{align*}
and
\begin{align*}
    \sum_{k=0}^{\infty}\frac{d_k}{\alpha_k+2} \int_{-\infty}^\infty|v_0|^{\alpha_k + 2} \, dx &= \sum_{k=0}^{\infty}\frac{d_kA^{\alpha_k + 2}}{\alpha_k+2} \int_{-\infty}^\infty\left|\frac{e^{\frac{-i\theta}{\langle x \rangle}}}{\langle x \rangle^n}\right|^{\alpha_k + 2} \, dx \\
    &= \sum_{k=0}^{\infty}\frac{d_kA^{\alpha_k + 2}}{\alpha_k+2} \int_{-\infty}^\infty\frac{1}{\langle x \rangle^{n(\alpha_k + 2)}} \, dx.
\end{align*}
From \eqref{E:NegE_requirement} we obtain the following condition
\begin{align*}
    2\theta^2 \int_{-\infty}^\infty \frac{x^2}{\langle x \rangle^{2n+6}}dx + \frac{n^2}{2} \int_{-\infty}^\infty \frac{x^2}{\langle x \rangle^{2n+4}}dx < \sum_{k=0}^{\infty}\frac{d_kA^{\alpha_k}}{\alpha_k+2} \int_{-\infty}^\infty\frac{1}{\langle x \rangle^{n(\alpha_k + 2)}} \, dx,
\end{align*}
which simplifies to (for $n(\alpha_k + 2) > 1$):
\begin{align*}
    2\sqrt{\pi}\theta^2 \frac{\Gamma\left(\frac{3}{2} + n\right)}{\Gamma\left(3 + n\right)} + \frac{n^2\sqrt{\pi}}{4} \frac{\Gamma\left(\frac{1}{2} + n\right)}{\Gamma\left(2 + n\right)} < \sqrt{\pi}\sum_{k=0}^{\infty}\frac{d_kA^{\alpha_k}}{\alpha_k+2} \frac{\Gamma\left(\frac{(n(\alpha_k + 2)) -1}{2}\right)}{\Gamma\left(\frac{n(\alpha_k + 2)}{2}\right)}.
\end{align*}
\textbf{Exponential case:}  Fix $n=2$ and choose $d_k = \frac{1}{k!}, \alpha_k  = k$ (which corresponds to the criteria of Theorem \ref{NegE_Bup_Exp}.), we have
\begin{align}\label{Weighted_specific_example}
    \frac{5\pi\theta^2+ 8\pi}{64} < \sqrt{\pi}\sum_{k=0}^{\infty}\frac{A^k}{k!(k+2)} \frac{\Gamma\left(\frac{3}{2} + k\right)}{\Gamma\left(3 + k\right)}.
\end{align}
\begin{figure}[htb!]
    \centering
    \begin{minipage}{.5\textwidth}
        \centering    \includegraphics[width=0.5\linewidth]{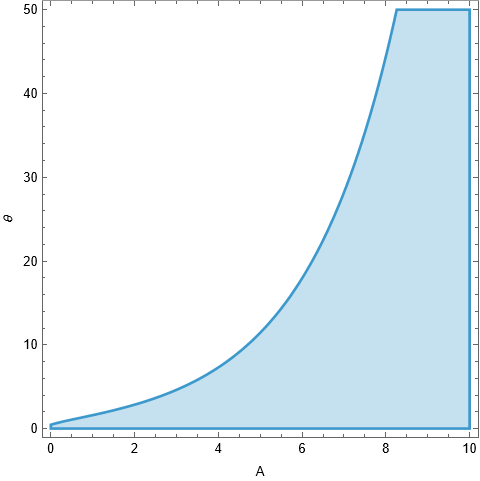}
        \captionof{figure}{\footnotesize Visualization of criteria \eqref{Weighted_specific_example}, where $\theta \in (0,50)$ and $A \in (0,10)$. The inequality \eqref{Weighted_specific_example} holds inside the blue area.}
        \label{fig:Weighted_specific_example}
    \end{minipage}%
    \begin{minipage}{.5\textwidth}
        \centering    \includegraphics[width=0.75\linewidth]{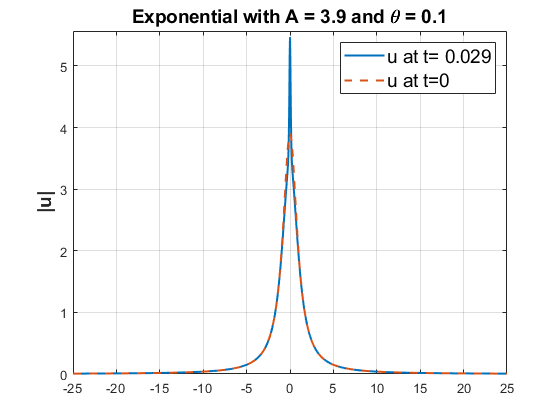}
        \captionof{figure}{\footnotesize Blow-up profile with $A = 3.9$ and $\theta = 0.1$, where the initial condition is of the form $u_0 = A\frac{e^{\frac{-i\theta}{\langle x \rangle}}}{\langle x \rangle^n}.$ }
        \label{fig:Blowup_exp_weighted}
    \end{minipage}
\end{figure}

Note that the RHS of \eqref{Weighted_specific_example} can be interpreted as a function of $A$ that grows as fast as an exponential. The LHS only grows as fast as $\theta^2$, thus, one can choose parameters $\theta$ and $A$ such that \eqref{Weighted_specific_example} holds. Numerically, one can achieve an estimated range (instead of an infinite sum we take it up to $N = 100$) that verifies this assertion, see Figure \ref{fig:Weighted_specific_example}. In Figure \ref{fig:Blowup_exp_weighted} we show an example of a solution with parameters $\theta=0.1$ and $A=3.9$. For more numerical simulations involving the exponential and combined nonlinearity, see \cite[Section 6]{RRR2024}.

\newpage
\subsection{Positive energy examples}

{\subsubsection{Gaussian}\label{S:PosE_Gaussian}
Let $u(x)  = A e^{-x^2}$ with $A > 0$.
Calculating the mass and variance, we get
\begin{equation}\label{E:PosE_Mass_Variance}
    M = \frac{A^2\sqrt{\pi}}{\sqrt{2}} \mbox{ and } V = \frac{A^2\sqrt{\pi}}{2^{\frac{5}{2}}}
\end{equation}
To show a clear example of the blow-up criteria in Theorem \ref{DR_Bup}, for this kind of initial condition, we fix $p =5$, $q = 7$, $\epsilon_1 = -1$ and $\epsilon_2 = 1$ in \eqref{CombinedNLS}. Then the energy is given by
\begin{equation}\label{E:PosE_p5q7}
    E = \frac{A^2\sqrt{\pi}}{2\sqrt{2}} + \frac{A^6\sqrt{\pi}}{6\sqrt{6}} - \frac{A^8\sqrt{\pi}}{16\sqrt{2}}.
\end{equation}

 Then, for real-valued data, the criteria in Corollary \ref{Cor:RealData} (for $p=5$ and $q = 7$) is 
\begin{equation}\label{Pos_E_bup_realdata}
   V[u_0] < \frac{M^2\left[1 + \frac{4|\epsilon_1|M^2}{3C^*}\right]}{24E}, 
\end{equation}
with $C^*$ as defined in \eqref{DRIneqConst}.

Then substituting \eqref{E:PosE_p5q7} and \eqref{E:PosE_Mass_Variance} into \eqref{Pos_E_bup_realdata}, we obtain (see Figures \ref{fig:p5q7_criteria} and \ref{fig:Blowup_posE_gaussain}) 

\begin{figure}[htb!]
    \centering
    \begin{minipage}{.45\linewidth}
        \centering    \includegraphics[width=0.5\linewidth]{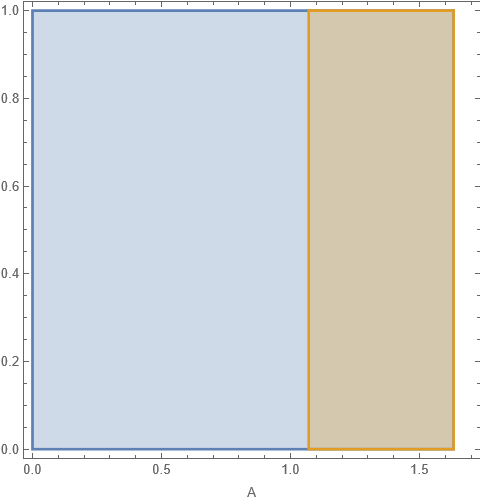}
        \captionof{figure}{\footnotesize The blue area shows the range, for which $A$ ensures positive energy (i.e., \eqref{E:PosE_p5q7} $> 0$), while the orange area (which covers the blue and extends further out) shows the range of $A$ such that \eqref{Pos_E_bup_realdata} holds. Notice the orange region has overlap.  }
        \label{fig:p5q7_criteria}
    \end{minipage}%
    \begin{minipage}{.45\textwidth}
        \centering    \includegraphics[width=0.8\linewidth]{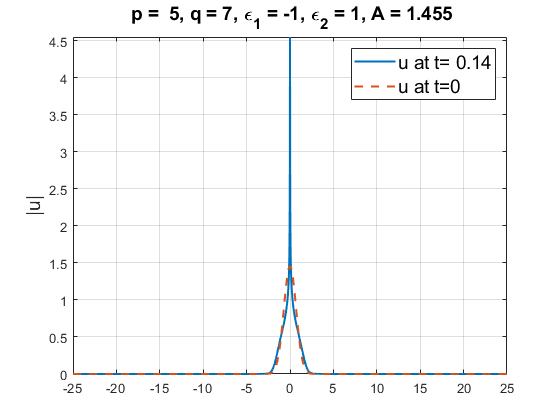}
        \captionof{figure}{\footnotesize Blow-up profile with $A = 1.455$ where the initial condition is of the form $u_0 = Ae^{-x^2}.$ }
    \label{fig:Blowup_posE_gaussain}
    \end{minipage}
\end{figure}

}
{

\subsubsection{Polynomial decaying data }\label{S:PosE_Polynomial}
In this example we take initial data $u(x) = \dfrac{A}{1+x^2}$, where $A > 0$. Deriving the mass, energy and variance, we get
\begin{equation}\label{M_V_poly}
    M = \frac{A^2\pi}{2} = V, \quad
\end{equation}

For this example, which corresponds to the criteria in Theorem \ref{DR_Bup}, we take $\epsilon_1 = -1$, $\epsilon_2 = 1$, $p = 5$ and $q = 7$, then the energy is given by
\begin{equation}\label{E_poly}
    E = \frac{A^2\pi}{8}\left( 1 + \frac{21A^{4}}{64} - \frac{429A^{6}}{2048}\right).
\end{equation}

Substituting \eqref{M_V_poly} and \eqref{E_poly} 
into \eqref{Pos_E_bup_realdata}, we obtain (See Figures \ref{fig:poly_criteria} and \ref{fig:Blowup_posE_weighted})

\begin{figure}[htb!]
    \centering
    \begin{minipage}{.45\textwidth}
        \centering    \includegraphics[width=0.5\linewidth]{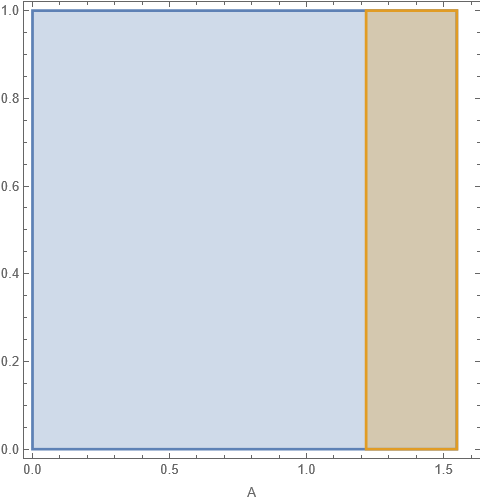}
        \captionof{figure}{\footnotesize The blue area shows the range for which $A$ ensures positive energy (i.e., \eqref{E_poly}$>0$), while the orange area (on top of blue) shows the range of $A$ such that \eqref{Pos_E_bup_realdata} holds.  }
        \label{fig:poly_criteria}
    \end{minipage}%
    \begin{minipage}{.45\textwidth}
        \centering    \includegraphics[width=0.8\linewidth]{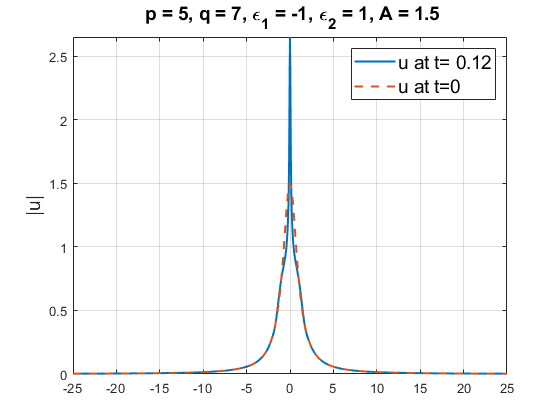}
        \captionof{figure}{\footnotesize Blow-up profile with $A = 1.5$ where the initial condition is of the form $u_0 = \frac{A}{1+x^2}.$ }
    \label{fig:Blowup_posE_weighted}
    \end{minipage}
\end{figure}

}

{\bf Acknowledgments.}
A.D.R. was partially supported by NSF grants DMS-2055130, 2221491, 2452782. The author would like to thank Svetlana Roudenko for numerous insightful conversations and suggestions related to this work.
\smallskip

{\bf Conflict of Interest:} The author declares that they have no conflicts of interest.
\smallskip

{\bf Data Availability}

Data sharing is not applicable to this article, as no data sets were generated or analyzed. 

\bigskip

{\bf ORCID}

{\it Alex D. Rodriguez}
\qquad {https://orcid.org/0000-0003-4382-7430}

\bibliographystyle{abbrv} 
\bibliography{main.bib}
\end{document}